\pdfoutput=1
\RequirePackage{amsmath}
\RequirePackage{fix-cm}
\RequirePackage{rotating}

\documentclass[11pt]{scrartcl}
\usepackage{amsthm}
\usepackage{thmtools, thm-restate}
\usepackage[title,header]{appendix}		

\newtheorem{theorem}{Theorem}[section]
\newtheorem{proposition}[theorem]{Proposition}
\newtheorem{lemma}[theorem]{Lemma}

\newtheorem{corollary}[theorem]{Corollary}

\usepackage{booktabs}
\newcommand{\head}[1]{%
	\begin{tabular}[m]{@{}c@{}}
		#1
	\end{tabular}%
}

\usepackage{rotating}
\usepackage[table]{xcolor}

\usepackage{amsfonts,amssymb}
\usepackage{graphicx}
\usepackage{commath,mathtools,bbm}	
\usepackage{nicefrac}
\usepackage{epstopdf}
\usepackage{url}
\usepackage[mathscr]{euscript}
\usepackage{enumitem}

\usepackage{thmtools, thm-restate}
\usepackage[ruled,vlined,linesnumbered]{algorithm2e} % vlined: no end but vertical line, lined -> with end
\SetKwRepeat{Do}{do}{while} % repeat until

\usepackage{footnote}
\usepackage{longtable,booktabs}
\usepackage{csvsimple,pgfkeys}

\usepackage[colorinlistoftodos,prependcaption]{todonotes}

\usepackage{scalerel}
\usepackage{tikz}
\usetikzlibrary{svg.path}
\usetikzlibrary{positioning}

\usepackage{pgfplots}

\definecolor{orcidlogocol}{HTML}{A6CE39}
\tikzset{
	orcidlogo/.pic={
		\fill[orcidlogocol] svg{M256,128c0,70.7-57.3,128-128,128C57.3,256,0,198.7,0,128C0,57.3,57.3,0,128,0C198.7,0,256,57.3,256,128z};
		\fill[white] svg{M86.3,186.2H70.9V79.1h15.4v48.4V186.2z}
		svg{M108.9,79.1h41.6c39.6,0,57,28.3,57,53.6c0,27.5-21.5,53.6-56.8,53.6h-41.8V79.1z M124.3,172.4h24.5c34.9,0,42.9-26.5,42.9-39.7c0-21.5-13.7-39.7-43.7-39.7h-23.7V172.4z}
		svg{M88.7,56.8c0,5.5-4.5,10.1-10.1,10.1c-5.6,0-10.1-4.6-10.1-10.1c0-5.6,4.5-10.1,10.1-10.1C84.2,46.7,88.7,51.3,88.7,56.8z};
	}
}

\newbox{\myorcidaffilbox}
\sbox{\myorcidaffilbox}{\large\mbox{\scalerel*{
			\begin{tikzpicture}[yscale=-1,transform shape]
			\pic{orcidlogo};
			\end{tikzpicture}
		}{|}}}
\newcommand\orcidicon[1]{\href{https://orcid.org/#1}{\usebox{\myorcidaffilbox}}}

\usepackage{authblk} %For Multiple authors

\DeclareMathOperator{\vol}{vol}
\DeclareMathOperator{\degree}{deg}
\DeclareMathOperator{\diag}{diag}

\newcommand{\h}[1][G]{h(#1)}
\newcommand{\hvol}[1][G]{h_{\vol}(#1)}
\newcommand{\sparse}[1][G]{\phi(#1)}
\newcommand{\mincut}[1][G]{\zeta_{\min}(#1)}

\newcommand{\trace}{\operatorname{tr}}

\newcommand{\F}{\mathcal{F}}
\newcommand{\FF}{\widetilde{\F}}
\newcommand{\real}{\mathbb{R}}

\newcommand{\inprod}[2]{\ensuremath{\langle #1, #2 \rangle}}
\renewcommand{\vec}[1]{\ensuremath{#1}}	
\newcommand{\mat}[1]{\ensuremath{#1}}

\definecolor{lightgray1}{gray}{0.90}

\usepackage[hypertexnames=false]{hyperref}
\hypersetup{colorlinks=true,citecolor=blue,linkcolor=red,urlcolor=blue,breaklinks=true}
\usepackage[nameinlink]{cleveref}

\crefname{equation}{}{}%{eq.}{eqs.}
\crefname{section}{\textsection}{\textsection}
\Crefname{algocf}{Algorithm}{Algorithms} % does not work!
\AtBeginEnvironment{appendices}{\crefalias{section}{appendix}\crefalias{subsection}{appendix}}
\Crefname{secinapp}{Appendix}{Appendices}
\begin{document}
	
	\title{Edge expansion of a graph: SDP-based computational strategies
		\thanks{A proceedings article containing part of this work appeared in~\cite{isco-paper}. This article now contains full and complete proofs, an additional algorithm based on a Discrete Netwon-Dinkelbach method, and more computational results.\\
			This research was funded in part by the Austrian Science Fund (FWF) [10.55776/DOC78]. For open access purposes, the authors have applied a CC BY public copyright license to any author-accepted manuscript version arising from this submission.
		}%
	}
	
	\author[1]{Akshay Gupte~\orcidicon{0000-0002-7839-165X}}
	\author[2]{Melanie Siebenhofer~\orcidicon{0000-0002-9101-834X}}
	\author[3]{Angelika Wiegele~\orcidicon{0000-0003-1670-7951}}
	
	\affil[1]{University of Edinburgh \& Maxwell Institute for Mathematical Sciences, UK,
		\href{mailto: akshay.gupte@ed.ac.uk}{akshay.gupte@ed.ac.uk}}
	\affil[2]{Alpen-Adria-Universit\"{a}t Klagenfurt, Austria,
		\href{mailto:melanie.siebenhofer@aau.at}{melanie.siebenhofer@aau.at}}
	\affil[3]{Alpen-Adria-Universit\"at Klagenfurt,
		Austria \& Universit\"at zu K\"oln, Germany, 
		\href{mailto:angelika.wiegele@aau.at}{angelika.wiegele@aau.at}}

	\date{\today}
	
	\maketitle

	\begin{abstract}
		Computing the edge expansion of a graph is a famously hard combinatorial problem for which there have been many
		approximation studies. We present two variants of exact algorithms using semidefinite programming (SDP)
		to compute this constant for any graph. The first variant uses the SDP relaxation first to reduce the search
		space considerably. 
		The problem is then transformed into instances of max-cut problems which are solved with an SDP-based
		state-of-the-art solver.
		Our second variant to compute the edge expansion uses Dinkelbach's algorithm for fractional programming. This is, we have to solve a parametrized optimization problem and again we use semidefinite programming to obtain solutions of the parametrized problems.
                Numerical results demonstrate that with our algorithms
		one can compute the edge expansion on graphs up
		to~400 vertices in a routine way, including instances
		where standard branch-and-cut solvers fail.
                To the best of our knowledge, these are the first SDP-based solvers for computing the edge expansion of a graph.
		
		\medskip
		\textbf{Keywords:} Edge expansion, Cheeger constant, bisection problems, semidefinite programming, parametric submodular minimization
		
	\end{abstract}

	\section{Introduction}
	\label{sec:introduction}
	
	Let $G = (V,E)$ be a simple connected graph on $n \ge 3$ vertices %, none of which are isolated, 
	and with $m$ edges.
	A \emph{cut} $(S,S^{\prime})$, for any $\emptyset\neq S\subset V$ and $S' = V \setminus S$, in $G$ is a partition of its vertices.
	The \emph{(unweighted) edge expansion}, also called the \emph{Cheeger constant} or \emph{isoperimetric number} or \emph{sparsest cut}, of $G$ is a ratio that measures the relative number of edges across any vertex partition. 
	It is defined as 
	\begin{align*}
	\h &= \min_{S} \left\{ \frac{|\partial S|}{\min\{|S|, |S^{\prime}|\}}:  \emptyset\neq S\subset V\right\}\\
	&= \min_{S} \left\{ \frac{|\partial S|}{|S|}: S\subset V,\, 1 \le |S| \le \frac{n}{2}\right\},
	\end{align*}
	where $\partial S = \{(i,j)\in E : i \in S,\, j\in S^\prime \}$ is the cut-set associated with any vertex subset~$S \subset V$, and $S^{\prime}=V\setminus S$. This constant is positive if and only if the graph is connected, and the exact value tells us that the number of edges across any cut in $G$ is at least $\h$ times the number of vertices in the smaller partition. Another version of edge expansion that accounts for vertex degrees is called the \emph{conductance} of a graph. It is defined as
	\begin{align*}	\label{def:hvol}
	\begin{split}
	\hvol &= \min_{S} \left\{ \frac{|\partial S|}{\min\{\vol(S), \vol(S^{\prime}) \}}: \emptyset\neq S\subset V\right\}\\
	&= \min_{S} \left\{\frac{|\partial S|}{\vol(S)}: S \subset V,\, 1 \le \vol(S) \le m\right\},
	\end{split}
	\end{align*}
	where $\vol(S) = \sum_{v\in S}\degree(v)$,
	and the second equality is due to $\vol(S) + \vol(S^{\prime}) = 2m$.

	Edge expansions arise in the study of expander graphs, for which there is a rich body of literature with applications in network science, coding theory, cryptography, complexity theory, cf.~\cite{sarnak2004communications,hoory2006expander,goldreich2011basic}. A graph with $\h \ge c$, for some constant $c > 0$, is called a $c$-expander. A graph with $\h < 1$ is said to have a bottleneck since there are not too many edges across it. A threshold for good expansion properties is having $\h \ge 1$, 
	which is desirable in many of the above applications.
	The famous Mihail-Vazirani conjecture \cite{mihail1992expansion,feder92matroids} in polyhedral combinatorics claims that the graph (1-skeleton) of any
	0/1-polytope has edge expansion at
	least~1. This has been proven to be true for several combinatorial polytopes \cite{mihail1992expansion,kaibel2004expansion} and bases-exchange graphs of matroids \cite{Anari2019}, and a weaker form was established recently for random 0/1-polytopes \cite{leroux2023randompolytopes}. Lattice polytopes were constructed in \cite{gupte2019dantzigfigures} with the property that in every dimension their graphs lie on the threshold of being good expanders (i.e., $\h = 1$).

	Computing the edge expansion is related to the \emph{uniform sparsest
		cut} problem which asks for computing a cut in the graph with the
	smallest sparsity, where sparsity is defined as the ratio of the size
	of the cut to the product of the sizes of the two partitions, 
	\begin{align*}
	\sparse &= \min_{S} \left\{ \frac{|\partial
		S|}{|S||S^{\prime}|}: \emptyset\neq S\subset V\right\}\\
	&= \min_{S}\,\left\{\frac{|\partial
		S|}{|S||S^{\prime}|}: S\subset V,\, 1 \le |S| \le \frac{n}{2}\right\}. 
	\end{align*}
	Since $\nicefrac{n}{2} \le |S^{\prime}| \le n$, it holds that
	$|S||S^{\prime}| \le n\cdot |S| \le 2 |S||S^{\prime}|$, and hence 
	$\h  \le n \cdot \sparse \le 2\h$, 
	which implies that the edge expansion problem is related to the sparsest cut problem up to a constant factor of 2. In particular, any cut $(S,S^{\prime})$ that is $\alpha$-approx for $\sparse$ (resp.~$\h$) is a $2\alpha$-approx for $\h$ (resp.~$\sparse$), because $|\partial S|/|S| \le n|\partial S|/(|S||S'|) \le  \alpha \cdot n\cdot \sparse \le 2\alpha \cdot\h$.
	
	There are polynomial reductions between $\h, \hvol$ and $\sparse$ and they are all NP-hard to compute \cite{Leighton99logn}, in contrast to the minimum-cut of a
	graph which can be computed in polynomial time. Hence, almost all of the literature on edge expansion is devoted to finding good theoretical bounds. These are generally associated with the eigenvalues of the Laplacian matrix of the graph and form the basis for the field of spectral graph theory (see the monograph \cite{chung1997spectral}). There have also been many approximation studies on this topic \cite{Leighton99logn,arora2009expander,nachmias2010testing,raghavendra2010graph}, and semidefinite optimization (SDP) has been a popular tool in this regard. The best-known approximation for $\sparse$ is the famous $\mathcal{O}(\sqrt{\log n})$ factor by Arora et al.~\cite{arora2009expander} which improved upon the earlier $\mathcal{O}({\log n})$-approximation~\cite{Leighton99logn}. The analysis is based on an SDP relaxation with triangle inequalities and uses metric embeddings and concentration of measure results. Meira and Miyazawa~\cite{Meira2011} developed a branch-and-bound algorithm for computing $\sparse$ using SDP relaxations and SDP-based heuristics. Recall that $\sparse$ is related to $\h$ in the approximate sense (up to a factor~2) but not in the exact sense. To the best of our knowledge, there is no exact solution algorithm for $\h$.

	\paragraph{Contribution and outline}

	We adopt mathematical programming approaches for numerical computation of $\h$.
	All our approaches make use of tight bounds obtained via semidefinite programming.
	The first algorithm works in two phases. 
	In the first phase, we split the problem into
	subproblems and by computing lower and upper bounds for these
	subproblems, we can exclude a significant part of the search space.
	In the second phase, we either solve the remaining subproblems to
	optimality or until a subproblem can be pruned due to the bounds.
	For the second phase, 
	we transform each subproblem into an instance of a
	max-cut problem and compute the maximum cut using an SDP-based solver.
	
	The second algorithm we implement uses the idea of Dinkelbach's algorithm
	to solve fractional optimization problems.
	The main concept of this algorithm is to iteratively solve linearly
	constrained binary quadratic programs. We solve these problems again by
	transforming them into instances of max-cut and using an SDP-based
	solver to compute the maximum cut. 
	
	We perform numerical experiments on different types of instances which
	demonstrate the effectiveness of our approaches. To the best of our
	knowledge, no other algorithms are capable of computing the edge
	expansion for graphs with a few hundred vertices.

	The rest of the paper is structured as follows. 
	In \Cref{sec:qp_and_sdp_formulation} we formulate the problem as a mixed-binary quadratic program and present an SDP relaxation.
	\Cref{sec:fixingk} investigates a related problem, namely the
	$k$-bisection problem. We introduce lower and upper bounds and describe how the $k$-bisection
	problem can be solved by transforming it to a max-cut problem. 
	The first algorithm (relying on the $k$-bisection problem) for
	computing $\h$ is introduced 
	in \Cref{sec:split-and-bound}, and another algorithm (following
	Dinkelbach's idea) in \Cref{sec:dinkelbach}.
	The
	performance of all algorithms is demonstrated in
	\Cref{sec:numericalresults}, followed by conclusions in \Cref{sec:summary}.
	
	\paragraph{Notation}
	The set of $n\times n$ real symmetric matrices is denoted by ${\mathcal S}^n$. The positive semidefiniteness condition for $\mat{X} \in \mathcal{S}^{n}$ is written as $\mat{X} \succeq 0$. 
	The trace of $\mat X$ is written as $\trace(\mat X)$ and defined as the sum of its diagonal elements. 
	The trace inner product for~$\mat X, \mat Y \in {\mathcal S}^{n}$ is defined as
	$\inprod{\mat X}{\mat Y}\;=\; \trace (\mat X \mat Y)$ and
	the operator $\diag(\mat X)$ returns the main diagonal of matrix
	$\mat X$ as a vector. The vector of all ones is $\vec e$ and the matrix of all ones is $\mat{J} = \vec e \vec e^\top$. 
	
	For an $n$-vertex graph $G=(V,E)$, the adjacency matrix is a binary matrix $\mat{A}\in\mathcal{S}^{n}$ having $A_{ij}=1$ if and only if~$(i,j)\in E$, and the degree matrix is a $n\times n$ positive diagonal matrix $\mat{D}$ having $D_{ii}$ equal to the degree of vertex $i\in V$. The Laplacian matrix is $\mat{L} = \mat{D} - \mat{A}$, and thus has its nonzero entries as~$L_{ii} = \degree(i)$ and $L_{ij} = -1$ for $(i,j) \in E$.
	The minimum cut in $G$ is defined as $\mincut = \min_{\emptyset\neq S\subset V}\, \abs{\partial S}$.

	\section{Formulations and SDP relaxations} \label{sec:qp_and_sdp_formulation}
	To write an algebraic optimization formulation for cut problems in graphs, we represent a cut $(S,S')$ in $G$ by its incidence vector $\vec \chi^{S}\in\{0,1\}^{n}$ which
	has~$\chi^{S}_{i}=1$ if and only if~$i\in S$. The cut function is the size of a cut-set, also called the value of the cut, and is equal to
	\[
         |\partial S| = \sum_{(i,j)\in E}\big(\chi^{S}_{i} - \chi^{S}_{j}\big)^{2} = \big(\vec \chi^{S}\big)^{\top} \mat{L} \vec \chi^{S} \ .
	\]
        Any binary vector $\vec x \in \{0,1\}^n$ represents a cut in this graph. 
        Denote the set of all cuts with $S$ containing at least one vertex and at most half of the vertices by
	\begin{equation*}
	\F = \left\{\vec x \in \{0,1\}^n : 1 \leq \vec e^\top \vec x \leq \frac n 2 \right\} .  
	\end{equation*}
	Using the common expression $\vec x^\top \mat L \vec x$ for the cut function, the edge expansion problem is
	
	\begin{align}
	\label{hforms}
	\begin{split}
	\h &= \min_{\vec x} \; \Big\{\dfrac{\vec x^\top \mat L \vec x}{\vec e^\top \vec x } : \vec x\in\F \Big\} 
	\medskip\\
	&= \min_{\vec x,y} \; \Big\{y : \dfrac{\vec x^\top \mat L \vec x}{\vec e^\top \vec x } \le y,\, \vec x\in\F, y \in \mathbb{R}\Big\} 
	\medskip\\
	&= \min_{\vec x,y} \; \big\{y : \vec x^\top \mat L \vec x - y\, \vec e^\top \vec x  \le 0,\, \vec x\in\F, y \in \mathbb{R}\big\}. 
	\end{split}
	\end{align}
	
	The last formulation is a mixed-binary quadratically constrained problem (MIQCP). 
	Standard branch-and-cut solvers may require a large
	computation time with these formulations even for instances
	of small to medium size, as we will report in \Cref{sec:numericalresults}. 
	
	Although the focus of this paper is on computing $\h$, let us also mention for the sake of completeness that analogous formulations can be derived for the graph conductance (weighted edge expansion) $\hvol$ that was defined in~\cref{sec:introduction}, by optimizing over the set
	\[
	\F_{\vol} = \left\{\vec x \in \{0,1\}^n : 1 \leq \vec d^\top \vec x \leq m \right\} ,  
	\]
	where $\vec d = \diag(\mat D)$ is the vector formed by the vertex degrees. For example, the same steps as in \cref{hforms} yields the MIQCP
	\[
	\hvol = \min_{\vec x,y} \; \big\{y : \vec x^\top \mat L \vec x - y \, \vec d^\top \vec x  \le 0,\, \vec x\in\F_{\vol}\big\}. 
	\]
	
	\subsection{Semidefinite relaxations}

	A well-known lower bound for the edge expansion is the \emph{spectral bound}.
	It is based on the second smallest eigenvalue of the Laplacian matrix of the graph,
	namely $h(G) \geq \nicefrac{\lambda_{2}(\mat L)}{2}$.
	One way to derive this bound is by considering the following SDP relaxation
	\begin{equation}
	\label{eq:SDP-expansion}
	\begin{array}{llllll}
	h(G) \geq &\min_{\widetilde{\mat X}, k} ~ & \frac 1 k \inprod{L}{\widetilde{\mat X}} &~
	=~ \min_{\mat X} ~ & \inprod{L}{X} & \\[0.1em]
	
	&\text{s.t.}~ & \trace(\widetilde{\mat X}) = k 
	&\phantom{~=~~}\text{s.t.}~ & \trace(\mat X) = 1\\[0.1em]
	
	&& \inprod{\mat J}{\widetilde{\mat X}} = k^2
	&& 1 \leq \inprod{\mat J}{\mat X} \leq  \frac n 2 \\
	
	&& 1 \leq k \leq \frac{n}{2}
	&& \mat X \succeq 0, \\
	&& \widetilde{\mat X} \succeq 0
	\end{array}
	\end{equation}
	where $\tilde{\mat X}$ models $\vec x\vec x^\top$ 
	and~$k = \vec e^\top \vec x$.  We add the redundant constraint~$\inprod{\mat J}{\widetilde{\mat X}} = k^2$
	and relax $\vec x \in \{0,1\}^n$, $\widetilde{\mat X} = \vec x \vec x^\top$ to $ \widetilde{\mat X}\succeq \mat 0$
	to obtain the above SDP relaxation.
	To eliminate the variable $k$ in the second (equivalent) SDP formulation, we
	scale~$\mat X = \frac{1}{k}\widetilde{\mat X}$.

        \begin{proposition}
        The optimal solution of the second SDP in~\eqref{eq:SDP-expansion} is $\nicefrac{\lambda_{2}(\mat L)}{2}$.
	\end{proposition}
	\begin{proof}
        It is well-known
        (cf.~\cite[\textsection5.2.6]{boyd2004convex}) that existence
        of a weak Slater point (i.e., a feasible point satisfying only
        the nonlinear convex constraints strictly) is sufficient to
        achieve strong duality for a general convex optimization
        problem. For the second SDP in~\eqref{eq:SDP-expansion}, the
        diagonal matrix $\frac 1 n \mat I$ is a weak Slater point
        because it is positive definite and has its trace equal to 1
        which is also equal to the sum of all its entries. Hence,
        strong duality holds with the dual which is
		\begin{equation*}
		\max_{\vec v}\; \Big\{v_1 - \frac n 2 v_2 + v_3 : \mat L - v_1 \mat I + (v_2 - v_3) \mat J \succeq 0,\, v_2,v_3 \geq 0\Big\}.
		\end{equation*}
                The psd constraint on the matrix $\mat W = \mat L - v_1 \mat I + (v_2 -
        v_3) \mat J$ can be
        simplified after making two observations. First we have $\mat
        J = ee^\top$. Then we deduce, after using the eigendecomposition of the Laplacian $\mat L$, that because $\mat
        L$ has $\vec e$ as its eigenvector corresponding to the
        smallest eigenvalue 0, and $\mat I$ is the identity matrix, it
        follows that $\vec e$ is also an eigenvector of $\mat W$ with
        eigenvalue
                $0 - v_1 + e^\top e (v_2 - v_3) = -v_1 + n(v_2 - v_3)$. The other eigenvalues of~$\mat W$ are then~$\lambda_i(\mat L) - v_1$ for $2 \leq i \leq n$. Therefore, we can write the dual as
		\begin{equation*}
		\max_{\vec v}\; \Big\{v_1 - \frac n 2 v_2 + v_3  : n(v_2 - v_3) \geq v_1,\ \lambda_2(\mat L) \geq v_1,\ v_2,v_3 \geq 0\Big\},
		\end{equation*}
		which is a linear program with optimal solution $v_1 = \lambda_2(\mat L)$, $v_2 = \nicefrac{\lambda_2(\mat L)}{n}$ and $v_3 = 0$ and optimal value $\nicefrac{\lambda_2(\mat L)}{2}$.

	\end{proof}

	To strengthen the SDP relaxation~\cref{eq:SDP-expansion} we round down the upper bound to~$\lfloor\frac{n}{2}\rfloor$
	and add the following facet-inducing inequalities of the boolean quadric polytope \cite{padberg1989boolean} for $\widetilde{\mat X}$
	\begin{subequations}
		\begin{align}
		0 \leq \widetilde{ X}_{ij} & \leq \widetilde{ X}_{ii}\\
		\widetilde{ X}_{i\ell} + \widetilde{ X}_{j\ell} - \widetilde{ X}_{ij} & \leq \widetilde{ X}_{\ell\ell}\\
		\widetilde{ X}_{ii} + \widetilde{ X}_{jj} - \widetilde{ X}_{ij} &\leq 1\\
		\widetilde{ X}_{ii} +  \widetilde{ X}_{jj} +  \widetilde{ X}_{\ell\ell} -  \widetilde{ X}_{ij} - 
		\widetilde{ X}_{i\ell} -  \widetilde{ X}_{j\ell} & \leq 1,
		\end{align}
	\end{subequations}
	resulting in the following valid inequalities for~$\mat X$
	\begin{subequations}
		\label{eq:bqpsX}
		\begin{align}
		0 \leq { X}_{ij} & \leq { X}_{ii} \label{subeq:bqpsX-1}\\
		{ X}_{i\ell} + { X}_{j\ell} - { X}_{ij} & \leq { X}_{\ell\ell} \label{subeq:bqpsX-2}\\
		{ X}_{ii} + { X}_{jj} - {X}_{ij} &\leq 1 \label{subeq:bqpsX-3} \\ % \leq \frac 1 k 
		{X}_{ii} +  {X}_{jj} +  {X}_{\ell\ell} -  {X}_{ij} - 
		{X}_{i\ell} -  {X}_{j\ell}  & \leq 1\label{subeq:bqpsX-4} % \leq \frac{1}{k} 
		\end{align}
	\end{subequations}
	for all $1 \leq i,\ j,\ \ell \leq n$.
	Note, that in~\cref{subeq:bqpsX-3,subeq:bqpsX-4} we have to replace~$\frac{1}{k}$ in the rhs by its upper bound~1 in order to obtain a formulation without~$k$. Therefore, we cannot expect these inequalities to strengthen the SDP relaxation significantly.

	\subsection{Illustrative examples for motivation}
	
	We motivate our algorithm by considering the example of the graph of the grlex polytope, which is described in~\cite{gupte2019dantzigfigures}.  \Cref{tab:lowerbounds} compares different lower bounds on $\h$ for these graphs. The first column indicates the dimension of the polytope and the
	second column lists the number of vertices in the associated graph.
	The third column gives the edge expansion that is known to be one for
	these graphs in all dimensions~\cite{gupte2019dantzigfigures}.
	The spectral bound is displayed in the fourth column. 
	Column~5 lists the optimal value of the SDP relaxation~\cref{eq:SDP-expansion} 
	strengthened by inequalities \cref{eq:bqpsX} derived from the boolean
	quadric polytope.
	Column~6 displays a lower bound that is very easy to compute: the
	minimum cut of the graph divided by the largest possible size of the
	smaller set of the bipartition of the vertices, that is 
	$\lfloor \frac{n}{2}\rfloor$. 
	In the last column, the minimum of the lower bounds $\ell_k$ for~$1
	\le k \le \lfloor \frac n 2 \rfloor$ is listed with $\ell_k$ being a
	bound related to the
	solution of~\cref{eq:SDP-expansion} for $k$ fixed. The definition of
	$\ell_k$ follows in \Cref{sec:lower-bisection}.
	
	The numbers in the table show that some of these bounds are
	very weak, in particular, if the number of vertices increases.
	Interestingly, if we divide the edge expansion problem
	into~$\lfloor \frac n 2 \rfloor$ many subproblems with fixed
	denominator (as we did to obtain the numbers in column~6)
	the lower bound we obtain by taking the minimum over all SDP relaxations
	for the subproblems seems to be stronger than the other lower bounds
	presented in \Cref{tab:lowerbounds}.
	We will, therefore, take this direction of computing the edge
	expansion, namely, we will compute upper and lower bounds on the
	problem with fixed $k$. Using these bounds will allow to exclude a
	(hopefully) large number of potential sizes $k$ of the smaller
	partition. This will leave us with computing the maximum cut of a
	graph with fixed sizes of the partition $k$ and $n-k$ for a few values
	of $k$ only. 
	
	\begin{table}
		\centering
		\begin{tabular}{ll|c|cccc} \toprule
			\head{$d$~} & \head{$n$} & \head{$\h$} & \head{$\nicefrac{\lambda_2}{2}$} & \head{\cref{eq:SDP-expansion} \& \cref{eq:bqpsX}~} & \head{$\mincut / \lfloor \frac n 2\rfloor$~} & \head{$\min_k\,\ell_k$} \\ \midrule
			2    &     4    &1 &    1.0000    &    1.0000    &    1.0000    &    1.0000 \\ 
			3    &     7    &1 &    0.7929    &    0.8750    &    1.0000    &    1.0000 \\ 
			4    &    11    &1 &    0.6662    &    0.7857    &    0.8000    &    1.0000 \\ 
			5    &    16    &1 &    0.5811    &    0.7273    &    0.6250    &    1.0000 \\ 
			6    &    22    &1 &    0.5231    &    0.6875    &    0.5455    &    1.0000 \\ 
			7    &    29    &1 &    0.4820    &    0.6591    &    0.5000    &    1.0000 \\ 
			8    &    37    &1 &    0.4516    &    0.6379    &    0.4444    &    1.0000  \\
			\bottomrule
		\end{tabular}
		\caption{Comparison of lower bounds for graphs from the grlex polytope
			in dimension~$d$.}
		\label{tab:lowerbounds}
	\end{table}
	
	\section{Fixing the size $k$: Bisection problem}
	\label{sec:fixingk}
	If the size $k$ of the smaller set of the partition of an optimum
	cut is known,
	the edge expansion problem would result in a scaled bisection problem.
	That is, we ask for a partition of the vertices into
	two parts, one of size $k$ and one of size $n-k$, such that the number
	of edges joining these two sets is minimized. This problem is NP-hard~\cite{garey76npcomplete} and has the following formulations for any $k\in\{1,2,\dots,\lfloor \frac n 2\rfloor\}$,
	\begin{equation}\label{eq:formulationExact-fixedk}
	\begin{array}{rll}
	h_k = \frac 1 k&\min_{\vec x} \ &  \vec x^\top \mat L \vec x \\[0.2em]
	&\text{s.t.} & \vec e^\top \vec x =k \\
	&& \vec x \in \{0,1\}^n,
	\end{array}
	\end{equation}
	but standard branch-and-cut solvers can solve these in reasonable time only for small-sized graphs. 
	
	Since SDP-based bounds have been shown to be very strong for partitioning
	problems,
	cf.~\cite{Meijer2023,gpp-shudian,karisch-rendl,wokowicz-zhao1999}, we
	exploit these bounds with our first algorithm to compute the edge expansion.
	In the subsequent sections, we describe how to obtain lower and upper
	bounds on~$h_k$ (\Cref{sec:lower-bisection,sec:upper-bisection}).
	We then present in \Cref{sec:bisect2mc}, how to transform the bisection problem
	into an instance of a max-cut problem which is then solved using the
	state-of-the-art solver BiqBin~\cite{biqbin}. For completeness, a description
	of BiqBin is given in \Cref{sec:biqbin}.
	
	\subsection{SDP lower bounds for the bisection problem}\label{sec:lower-bisection}
	
	After squaring the linear equality constraint in problem~\cref{eq:formulationExact-fixedk} and employing standard lifting and relaxation techniques, we obtain the following SDP relaxation that is generally computationally cheap to solve,
	\begin{equation}\label{eq:cheapsdp}
	\begin{array}{rll}
	\ell_{\text{bisect}}(k) =& \min_{\mat X,\vec x} \ & \langle \mat L, \mat X \rangle \\[0.2em]
	&\text{s.t.} &  \trace(\mat X)  = k\\
	&& \langle \mat{J}, \mat X \rangle = k^2\\
	&& \diag(\mat X) = x\\
	&& \begin{pmatrix}
	1 & &\vec x^\top \\
	\vec x & & \mat X
	\end{pmatrix} \succeq 0.
	\end{array}
	\end{equation}
	Since the bisection for a given simple unweighted graph has to be an integer, we get the following lower bound on the scaled bisection~$h_k$,
	\begin{equation}	\label{eq:cheaplb}
	h_{k} \ge \ell_k = \frac{\lceil \ell_{\text{bisect}}(k) \rceil}{k} .
	\end{equation}

	There are several ways to strengthen the above relaxation of the bisection
	problem.
	In~\cite{wokowicz-zhao1999} a vector lifting SDP relaxation,
	tightened by non-negativity constraints, has been introduced.
	In our setting, this results in the following doubly non-negative
	programming (DNN) problem,
	\begin{align}\label{eq:dnn}
	\begin{split}
	\min_{\mat X} ~~&\inprod{L}{\mat X^{11}+\mat X^{22}} \\%[1ex]
	\textrm{s.t.}~~&\trace (\mat X^{11}) = k, ~ \inprod{\mat J}{\mat X^{11}} = k^{2} \\%[1ex]
	&\trace (\mat X^{22}) = n-k, ~ \inprod{\mat J}{\mat X^{22}} = (n-k)^{2} \\%[1ex]
	&\diag(\mat X^{12})=\vec{0},~
	\diag(\mat X^{21})=\vec{0},
	~ \inprod{\mat J}{\mat X^{12}+\mat X^{21}} = 2k(n-k) \\%[1ex]
	& 
	\mat X= \begin{pmatrix}
	1 & (\vec x^1)^\top & (\vec x^2)^\top \\
	\vec x^1 & \mat X^{11} & \mat X^{12} \\
	\vec x^2 & \mat X^{21} & \mat X^{22} \\
	\end{pmatrix} \succeq {0},
	~~\vec x^i=\diag(\mat X^{ii}), ~~ i=1,2\\%[1ex]
	&\mat X \ge 0,
	\end{split}
	\end{align}
	where $\mat X$ is a matrix of size $(2n+1) \times (2n+1)$. This relaxation can be further 
	strengthened by cutting planes from the Boolean Quadric
	Polytope. In particular, we want to add the inequalities
	\begin{equation}\label{eq:cuttingplane}
	X_{i\ell} +  X_{j\ell} \le  X_{\ell\ell} +  X_{ij}
	\end{equation}
	as Meijer et al.~\cite{Meijer2023} demonstrated that these
	inequalities are the most promising ones to improve the bound.
	
	The DNN relaxation~\cref{eq:dnn} cannot be solved by standard
	methods due to the large number of
	constraints. The additional cutting-planes~\cref{eq:cuttingplane}
	make the SDP relaxation extremely difficult to solve already for
	medium-sized instances. Meijer et al.~\cite{Meijer2023} apply facial reduction to the SDP
	relaxation which leads to a natural way of splitting the set of
	variables into two blocks. 
	Using an alternating direction method of multipliers (ADMM) provides
	(approximate) solutions to this relaxation even for 
	graphs with up to~1000~vertices.
	The steps to be performed in this ADMM algorithm result in projections
	onto the respective feasible sets. For projections onto polyhedra,
	Dykstra's projection algorithm is used.
	A careful selection of non-overlapping cuts, warm starts, and an
	intelligent separation routine are further ingredients of this
	algorithm in order to obtain an efficient solver for the
	SDP~\cref{eq:dnn} enhanced with inequalities~\cref{eq:cuttingplane}.
	A post-processing algorithm is also introduced to guarantee a valid lower bound. Using this algorithm, we can compute strong lower bounds for each $k$ with reasonable computational effort.

	\subsection{A heuristic for the bisection problem}\label{sec:upper-bisection}
	The graph bisection problem can also be written as a quadratic assignment problem (QAP)~\cite{deklerk2012bisect}. 
	To do so, we set the weight matrix~$\mat W$ to be the Laplacian matrix $\mat L$ of the graph and the distance matrix~$\widetilde{\mat D}$ to be a matrix with a top left block of size $k$ with all ones and the rest zero. 
	The resulting QAP for this weight and distance matrix is
	\begin{equation*}
	\min_{\pi \in \Pi_n} \  \sum_{i = 1}^n \sum_{j=1}^n W_{i,j} \widetilde D_{\pi(i),\pi(j)} = \min_{\pi \in \Pi_n} \  \sum_{i = 1}^k \sum_{j=1}^{k}  L_{\pi^{-1}(i),\pi^{-1}(j)} = kh_k.
	\end{equation*}
	In this formulation, the vertices mapped to values between~1 and~$k$ by the permutation~$\pi$ are chosen to be in the set of size~$k$ in the partition. To compute an upper bound~$u_k$ on~$h_k$, we can use any heuristic for the QAP and divide the
	solution by~$k$.
	
	Simulated annealing is a well-known heuristic to compute an upper bound for the QAP, we implement the algorithm introduced in~\cite{Burkard1984}. 
	We use a slightly different parameter setting that we determined via numerical experiments. That is, the initial temperature is set to $\nicefrac{k^2}{\binom{n}{2}}\cdot \operatorname{tr}(\mat L)$, the number of transformation trials at constant temperature is initially set to~$n$ and increased by a factor of~1.15 after each cycle, and the cooling factor is set to~0.7. 
	After every trial, we additionally perform a local search strategy to find the local minimum.

	Other well-performing heuristics for the QAP are tabu search, genetic algorithms, or algorithms based on the solution
	of the SDP relaxation like the hyperplane rounding algorithm.
	Some of these heuristics have the potential to be superior to
	simulated annealing.
	However, as we will see later in the study of our experiments, the
	bounds we obtain with the simulated annealing heuristic are almost
	always optimal for the size of instances we are interested in.

	\subsection{Transformation to a max-cut problem}\label{sec:bisect2mc}
	One variant to solve the $k$-bisection problem is to implement
	a branch-and-bound algorithm with the aforementioned lower and upper bounds
	as presented in~\cite{isco-paper}.
	We are going to use
	a different approach to solving the graph bisection problem. Namely, we
	transform it to a max-cut problem and then take advantage of  using a
	well-established and performant
	max-cut solver, e.g. the open source parallel solver BiqBin~\cite{biqbin}, see also \Cref{sec:biqbin}.
	To do so, we first need to transform the bisection problem into
	a quadratic unconstrained binary problem (QUBO). 
	
	\begin{lemma}
		Let $\tilde{\vec x} \in\{0,1\}^n$ such that $\vec e^\top \tilde{\vec x} = k$, 
		and choose $\mu_k$ such that $\mu_k > \tilde{\vec x}^\top \mat L\tilde{\vec x}$.
		Then 
		\[
		h_k = \frac{1}{k}\min_{\vec x} \;\big\{\vec x^\top (\mat L + \mu_k \mat{J}) \vec x -2\mu_k k\, \vec e^\top \vec x + \mu_k k^2: \vec x \in \{0,1\}^n\big\} .
		\]
	\end{lemma}
	\begin{proof}
		First note that \[ \vec x^\top (\mat L + \mu_k \vec e\vec e^\top)\vec x - 2\mu_k
		k\,\vec e^\top \vec x + \mu_k k^2 = \vec x^\top \mat L\vec x + \mu_k \|e^\top \vec x -
		k\|^2.\]
		Let $\vec x\in \{0,1\}^n$. Then we have 
		%For $\vec e^\top \vec x = k$ we have $\vec x^\top L\vec x + \mu_k \|\vec e^\top \vec x - k\|^2 = \vec x^\top \mat L\vec x$.
		%And if $\vec e^\top \vec x \not= k$, $\vec x^\top L\vec x + \mu_k \|e^\top \vec x - k\|^2 > \mu_k$.
		\begin{align*}
		&\vec x^\top \mat L\vec x + \mu_k \|\vec e^\top \vec x - k\|^2 = \vec x^\top \mat L\vec x &\text{ if } \vec e^\top\vec x = k,\\
		&\vec x^\top \mat L\vec x + \mu_k \|\vec e^\top \vec x - k\|^2 \ge \mu_k &\text{ if } \vec e^\top\vec x \not= k.
		\end{align*}
		Note that $\vec e^\top \vec x - k$ is integer for $\vec x \in \{0,1\}^n$.
		Hence, for any infeasible~$\vec x \in \{0,1\}^n$, the objective is greater
		than the given upper bound $\tilde{\vec x}^\top \mat L\tilde{\vec x}$, and therefore the minimum can only
		be attained for $\vec x \in \{0,1\}^n$ with $\vec e^\top \vec x = k$.
	\end{proof}
	
	Barahona et al.~\cite{Barahona1989} showed that any QUBO problem can be reduced to a max-cut problem by adding one additional binary variable. 
	In our context, this means the following. 
	
	\begin{corollary}
		Let $G=(V,E)$ and let $G'$ be the complete graph with vertex set
		$V \cup \{v_0\}$. Let the weights $c_{uw}$ on the edges of $G'$ be as follows.
		\begin{equation*}
		c_{uw} = \begin{cases}
		4\mu_k(n-2k) & \text{ if } u\in V(G) \text{ and } w = v_0\\
		4\mu_k - 1 & \text{ if } uw \in E(G)\\
		4\mu_k &\text{ if } uw \not\in E(G)
		\end{cases}
		\end{equation*}
		Then the minimum bisection of $G$ where one side of the cut has size~$k$ is
		equal to $\mathtt{offset} - \operatorname{max-cut}(G')$,
		where $\mathtt{offset} = 4\mu_k (n-k)^2$.
	\end{corollary}
	
	Since max-cut solvers can benefit from edge weights of the input
	graph being integer, a possible choice for~$\mu_k$ is an upper bound on the
	bisection problem plus~$\nicefrac 1 4$.
	Note that we choose~$\mu_k$ to be as small as possible by doing so.
	
	The formulation of the max-cut problem in $\pm 1$ variables additionally
	requires~$x_{v_0} = 1$ to hold. Because of the symmetry of the cut, we can omit this constraint. Due to our choice of the penalty parameter, it holds that on one side of
	the maximum cut, there are exactly~$k+1$ vertices, including vertex~$v_0$.
	These~$k$ vertices on the same side as~$v_0$
	are the vertices in the set of size~$k$ in the optimum of the bisection problem.
	
	To summarize, we can solve the bisection problem by solving a dense max-cut
	problem. With this, we now have all the tools needed for our new split \& bound approach. 
	
	\section{Split \& bound}
	\label{sec:split-and-bound}
	We now assemble the tools developed in \Cref{sec:fixingk} 
	to compute the edge expansion of a graph by splitting
	the problem into~$\lfloor \frac n 2 \rfloor$ many bisection problems.
	Since the bisection problem is NP-hard as well, we want to reduce the
	number of bisection problems we have to solve exactly as much as
	possible. To do so, we start with a pre-elimination of the
	bisection problems.
	This procedure aims to exclude subproblems unnecessary for
	the computation of the edge expansion of the graph.
	Computing the edge expansion by considering the remaining values of $k$ is summarized in
	Algorithm~\ref{alg:split-and-bound} below.
	We now explain the pre-elimination step and further ingredients of our
	algorithm.  
	
	\subsection{Pre-elimination}
	The size~$k$ of the smaller set of the partition can theoretically
	be any value from $1$ to~$\lfloor \frac{n}{2}\rfloor$.
	However, it can be expected that for some candidates, one can quickly
	check that the optimal solution cannot be attained for that $k$.
	As a first quick check, we use the cheap lower bound~$\ell_k$ obtained
	by solving the SDP~\cref{eq:cheapsdp} 
	in combination with the upper bound introduced in the
	\Cref{sec:upper-bisection}.
	We do not need to further consider values of~$k$ where the
	lower bound~$\ell_k$ of the scaled bisection problem is
	already above an upper bound on the edge expansion.
	A pseudo-code of this pre-elimination step
	is given in Algorithm~\ref{alg:pre-elimination}.
	
	\begin{algorithm} \label{alg:pre-elimination}
		%\KwIn{graph $G$}
		%\KwOut{set $\mathcal{I}$ of values $k$ needed, upper bound $u^* \geq h(G)$}
		\For{$k \in \{1,\dots, \lfloor \frac{n}{2}\rfloor\}$}{
			Compute an upper bound $u_k$ using the heuristic from \Cref{sec:upper-bisection}\;
			Compute the lower bound $\ell_k$ from \cref{eq:cheaplb} by solving the cheap SDP~\cref{eq:cheapsdp}\;}
		Global upper bound $u^* \coloneqq \min \big\{ u_k\colon 1 \leq k \leq \lfloor \frac{n}{2}\rfloor\big\}$\;
		\eIf{$\min_{k} \ell_k = u^*$}{$\mathcal{I} = \emptyset$, $h(G) = u^*$\;}
		{
			$\mathcal{I} \coloneqq \big\{ k \in \{ 1,\dots,\lfloor
			\frac{n}{2}\rfloor\} \colon \ell_k < u^* \big\}$\;}
		\Return $\mathcal{I}$, $u_k$ for $k \in \mathcal{I}$, $u^*$
		\caption{Pre-eliminate certain values of $k$}
	\end{algorithm}
	
	The hope is that many values of
	$k$ can be excluded from computing the edge expansion.
	Clearly, this heavily depends on the instance itself, as in the
	worst case, it might happen that for many different values of~$k$, the value of~$h_k$
	is close to the optimum.
	
	We can further reduce the number of candidates for $k$ by computing a tighter
	lower bound~$\tilde{\ell}_k$ by solving the DNN relaxation~\cref{eq:dnn} with additional cutting
	planes. In our implementation we do not compute 
	$\tilde{\ell}_k$ 
	as part of the pre-elimination but use the
	lower bound obtained from the solution in the root node of the max-cut solver.
	
	\paragraph{Impact of pre-elimination on sample instances}
	\Cref{fig:bounds-each-k-grevlex,fig:bounds-each-k} display the bounds associated with four different graphs. For the graph
	of the grevlex polytope in dimension~7, considering the bounds $u_k$
	and $\tilde{\ell}_k$ the only candidates for $k$
	where the optimal solution can be found are 12 and 14.
	For the grevlex polytope in dimension~8, the sizes~17 and~18 remain as
	the only candidates. Also for a graph associated to a randomly
	generated 0/1-polytope and to a network graph, about 2/3 of the
	potential values of $k$ can be excluded already by considering the
	cheap lower bound $\ell_k$.
	
	\begin{figure}[h!]
		\centering
		\includegraphics[width=0.95\linewidth]{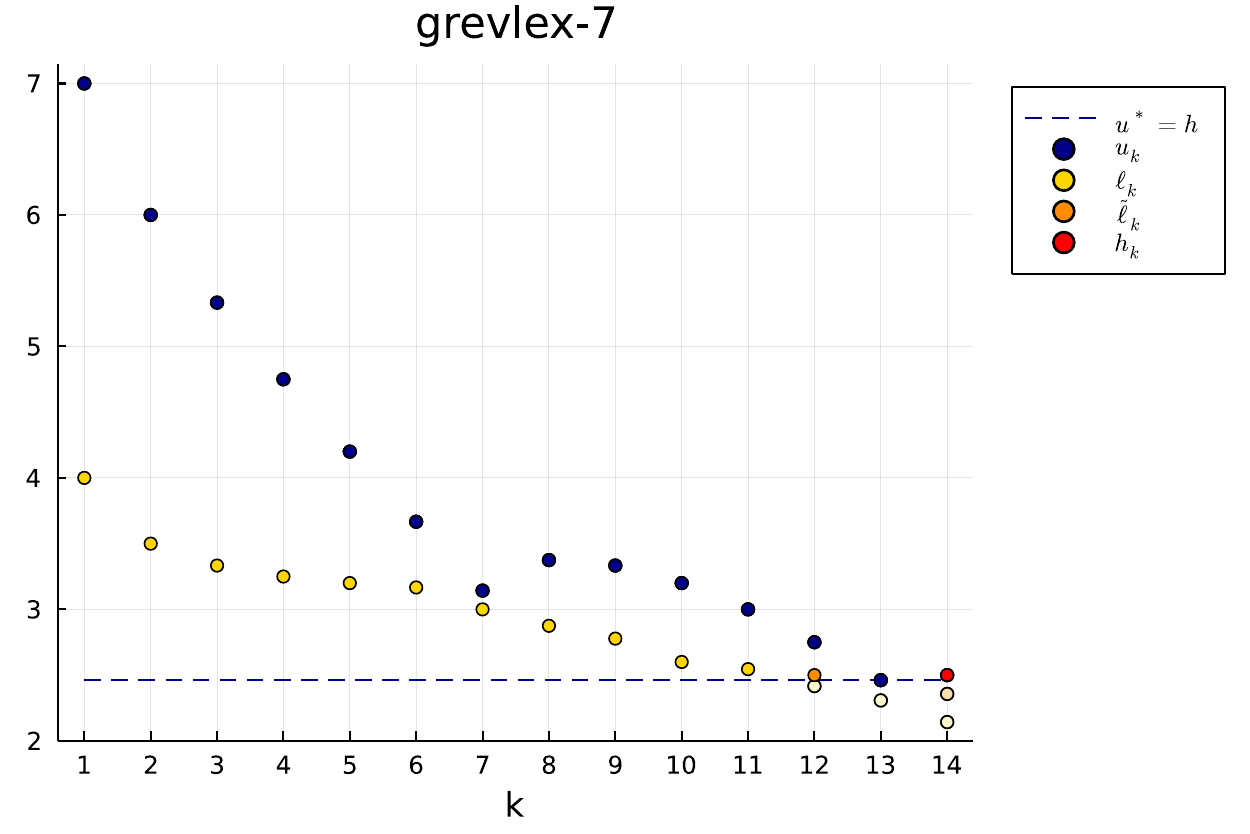}
		\includegraphics[width=0.95\linewidth]{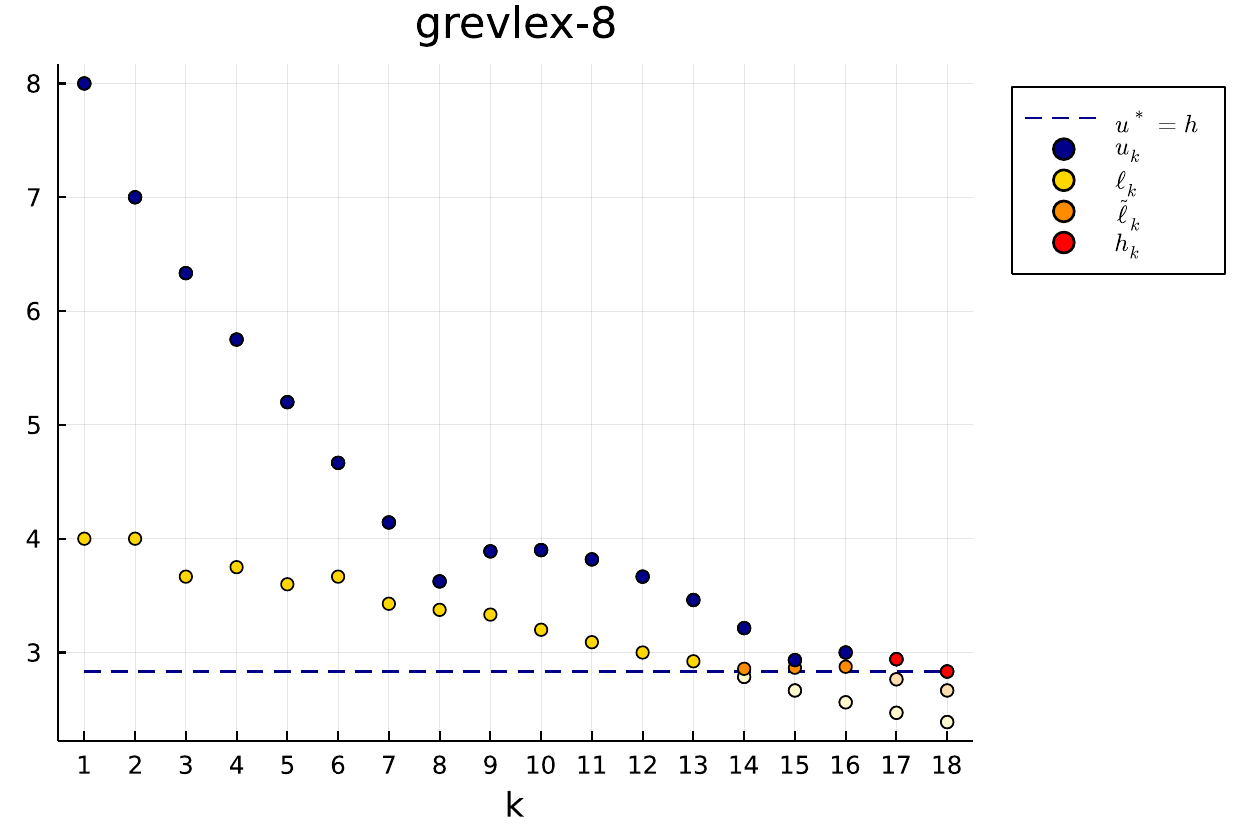}
		\caption{Lower and upper bounds for each $k$.}
		\label{fig:bounds-each-k-grevlex}
	\end{figure}
	\begin{figure}[h!]
		\includegraphics[width=0.95\linewidth]{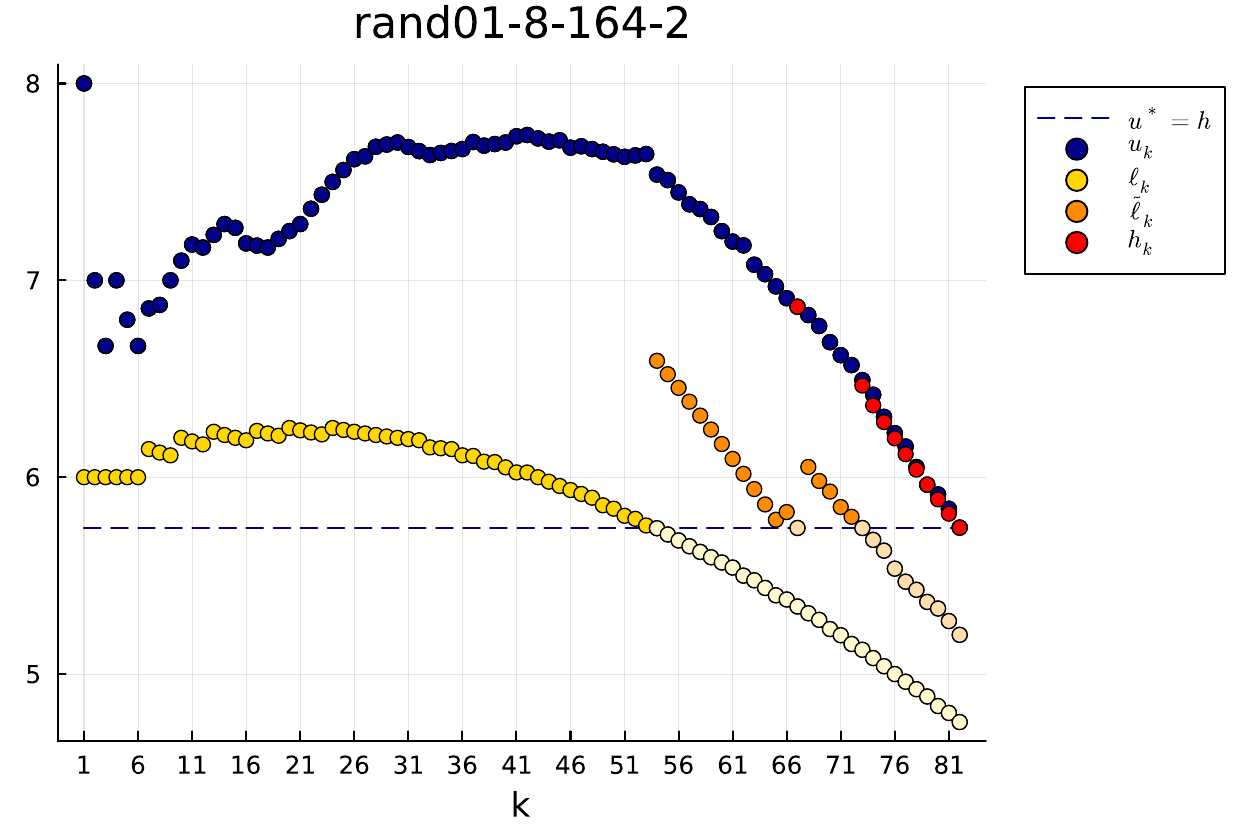}
		\includegraphics[width=0.95\linewidth]{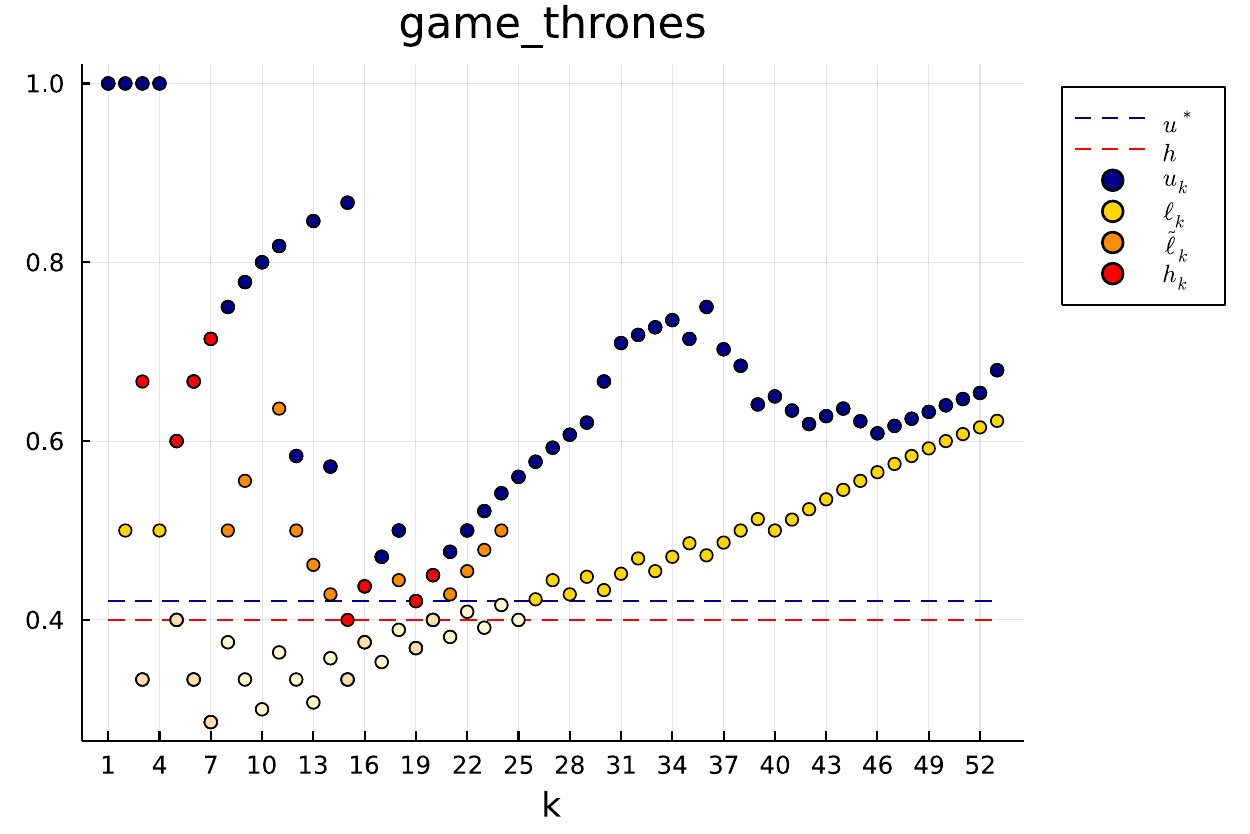}
		\caption{Lower and upper bounds for each $k$.}
		\label{fig:bounds-each-k}
	\end{figure}

	\subsection{Stopping exact computations early and updating $u^*$}
	All values of~$k$ that are not excluded in the pre-elimination step
	have to be further examined. 
	For those values we continue to run the max-cut solver
	to compute the scaled bisection~$h_k$. We can stop the
	branch-and-bound algorithm as soon as the lower bound (on the scaled $k$-bisection problem)
	of all open nodes in the branch-and-bound tree is larger or equal to~$u^*$.
	(Remember that~$u^*$ is an upper bound on the edge expansion of the graph
	but not necessarily an upper bound on $h_k$.)
	A simple way to implement this stopping criterion
	is to initialize the algorithm for the $k$-bisection
	problem with $\lceil u^* k \rceil$ as an
	``artificial'' upper bound.
	For the max-cut solver, this translates to an ``artificial'' lower bound
	of~$\mathtt{offset} - \lceil u^* k \rceil$.
	
	In case $h_k < u^*$, we can update $u^*$ which might lead to
	eliminating further values from~$\mathcal{I}$.
	This fact is also part of the motivation for the order of choosing~$k$ for
	computing~$h_k$, as described in the next section.

	\subsection{Order of selecting values $k$ from $\mathcal{I}$}
	We consider the order of computing $h_k$ in ascending order based on their
	upper bounds~$u_k$. The motivation for this choice is as follows.
	
	Remember that $u^* = \min_k u_k$ is a global upper bound on the edge
	expansion. The most promising values for $k$ to even further improve
	this bound are those with small~$u_k$. 
	Therefore, before starting the max-cut solver for the
	values $k$ left after pre-elimination, we do another 30~trials of
	simulated annealing for each of these 
	to hopefully further improve the upper bound.
	
	Moreover, we run the exact computation of $h_k$ in this order since also
	during the branch-and-bound algorithm, the upper bound~$u_k$ might drop
	further and this will improve the global upper bound $u^*$ most likely
	for those candidates with small $u_k$.
	
	An improvement of the upper bound~$u^*$ means that there is a
	possibility to further eliminate values $k$ from $\mathcal I$.
	But even for values $k$ that can not be eliminated, we obtain
	smaller artificial upper bounds and hence the computation of these
	bisection problems may be stopped earlier.
	
	We summarize all the steps in Algorithm~\ref{alg:split-and-bound}.

	\begin{algorithm} \label{alg:split-and-bound}
		$\mathcal{I}$, $u_k$ for $k \in \mathcal{I}$, $u^*$ $\leftarrow$ pre-elimination Algorithm~\ref{alg:pre-elimination}\;
		\For{$k \in \mathcal{I}$}{
			Run heuristic from \Cref{sec:upper-bisection} and update $u_k$\;
			\If{$u_k < u^*$}{$u^* \leftarrow u_k$\; update $\mathcal{I}$\;}}
		\For{$k \in \mathcal{I}$, consider $k$ in ascending order of $u_k$}{
			transform the instance to a max-cut instance\;
			compute $h_k$ using the max-cut solver, initialize the lower bound for max-cut as $\mathtt{offset} - \lceil u^*k \rceil$\;
			\If{$h_k < u^*$}{$u^* \leftarrow h_k$\; update $\mathcal{I}$}}
		$h(G) = u^*$\;
		\caption{Split \& bound}
	\end{algorithm}

	\subsection{Algorithmic verification of lower bound}
	We close this section by addressing the important consideration that we are not interested in the exact value of
	the edge expansion in some applications, but want to check whether certain values are valid lower bounds on $\h$.  A lower bound $c \le \h$, for some constant $c>0$, means that the graph is a $c$-expander. The value of this lower bound means that the graph expands by at least that much. This also arises in the context of the Mihail-Vazirani conjecture on 0/1-polytopes where one wants to check whether~$\h \ge 1$ where~$G$ is the graph of a 0/1-polytope. 
	
	Our split \& bound algorithm can also be used to verify a lower bound.
	
	\begin{proposition}
		Let $\upsilon$ be a given scalar and suppose we initialise Algorithm~\ref{alg:split-and-bound} with $u^{\ast} = \upsilon$. Then $\upsilon$ is a valid lower bound on $\h$ if and only if the algorithm terminates without updating $u^{\ast}$. 
	\end{proposition}
	\begin{proof}
		Assume we initialize $u^{\ast} = \upsilon$. 
		If we find a better upper bound (or some computed value~$h_k$
		is smaller than $\upsilon$), this is a certificate that
		the given value~$\upsilon$ is not a valid lower bound since we found
		a better solution.
		Otherwise, if the upper bound never gets updated, this means
		the provided bound is indeed a valid lower bound on the edge expansion
		of the graph.
	\end{proof}

	\section{Parametric optimization}
	\label{sec:dinkelbach}
	Another approach to compute $\h$ is following a discrete
	Newton-Dinkelbach algorithm.
	Dinkelbach~\cite{dinkelbach1967nonlinear} gave a general classical framework to solve (non)-linear fractional programs. 
	The program one aims to solve is $\min_{\vec x \in \F}\, f(\vec x)$, where the objective~$f$
	is a fraction of (non)-linear functions.
	In our case this is
	\begin{subequations}
		\begin{equation*}
		f(\vec x) = \frac{\vec x^\top \mat L \vec x}{\vec e^\top \vec x}, \quad \text{and} \quad \mathcal F = \Big\{\vec x \in \{0,1\}^n : 1 \leq \vec e^\top \vec x \leq \frac n 2 \Big\} .  
		\end{equation*}
		The main component of the algorithm is to form the following parametrized objective function
		\begin{equation*}
		g_{\gamma}(\vec x) = \vec x^\top \mat L \vec x - \gamma \vec e^\top \vec x,
		\end{equation*}
		and the corresponding parametrized optimization problem
		\begin{equation*}
		P(\gamma) = \min_{\vec x}\; \{ g_{\gamma}(\vec x): \vec x \in \F\}, \qquad \gamma \ge 0 .
		\end{equation*}
	\end{subequations}
	This problem then has the following useful properties. 
	%, which we prove using a specialised version of \citeauthor{dinkelbach1967nonlinear}'s arguments tailored to our problem.

	\begin{proposition} \label{prop:dinkelbach-sign}
		$P(0) = \mincut$ and $P$ is a strictly decreasing concave piecewise linear function over $\real_{+}$ whose unique root is equal to $h(G)$. Consequently,
		\[
		\h = \max_\gamma \big\{\gamma \colon g_{\gamma}(\vec x) \ge 0  \big\} = \min_\gamma \big\{\gamma \colon g_{\gamma}(\vec x) \le 0  \big\} .
		\]
	\end{proposition}
	
	\begin{proof}
		We have
		\begin{align*}
		P(0) &= \min_{\vec x} \big\{ \vec x^\top \mat L \vec x: \vec x \in \F\big\}\\
		&= \min_{S}\big\{ \lvert {\partial S}\rvert : \emptyset\neq S\subset V, |S| \le \nicefrac{n}{2}\big\} \\
		&= \min_{S}\big\{ \lvert {\partial S}\rvert : \emptyset\neq S\subset V\big\} \\
		&= \mincut 
		\end{align*}
		where the penultimate equality is from symmetry of the cut function $\abs{\partial S} = \abs{\partial S^{\prime}}$. Finiteness of $\F$ and linearity of $g_{\gamma}$ in $\gamma$ tells us that $P$ is the pointwise minimum of finitely many affine (in $\gamma$) functions, and so $P$ is a concave piecewise linear function. The strictly decreasing property was shown in \cite[Lemma~3]{dinkelbach1967nonlinear} for general nonlinear fractional problems.
	\end{proof}
	
	This implies that the edge expansion of a graph can be computed using a root-finding algorithm for the function $P$. One evaluation of $P$ for a given~$\gamma$ still means solving a binary quadratic problem with two linear inequalities. Hence, reducing this number of evaluations is crucial to compute $h(G)$ in reasonable time. There are several strategies to do so, such as binary search. Our approach is to evaluate~$P$ starting with~$\gamma_1$ equal to some good upper bound on $\h$ (in our experiments, we used our heuristic from \Cref{sec:upper-bisection}).
	We are already done if we have found the optimum with our heuristic, that is when~$P(\gamma_1) = 0$.
	Otherwise, there is some 
	$\vec x^{1} \in \F$ such that $g_\gamma(\vec x^{1}) < 0$ and therefore~$f(\vec x^{1}) < \gamma$.
	This means that~$f(\vec x^{1})$ is a better upper bound than~$\gamma_1$.
	Hence, we now set $\gamma_2 = f(\vec x^{1})$ and repeat until we find the optimum as described
	in Algorithm~\ref{alg:dinkelbach}. Since~$P(\gamma) < 0$ if and only if $\gamma > \h$, the stopping criterion is checking whether $P(\gamma) < 0$ at the current iterate.

	\begin{algorithm} \label{alg:dinkelbach}
		\KwIn{graph $G$, upper bound $\gamma_1 \geq h(G)$ from heuristic}
		\KwOut{edge expansion $h(G)$}
		$i = 1$\;
		\While{$P(\gamma_{i}) < 0$}{
			$\vec x_{i} \in \arg \min_{\vec x \in \F} g_{\gamma_{i}}(\vec x)$\;
			$\gamma_{i+1} = f(\vec x_i)$\;
			$i = i+1$\;}
		$h(G) = \gamma_i$\;
		\caption{Discrete Newton-Dinkelbach algorithm for edge expansion}
	\end{algorithm}

	The superlinear convergence rate of Dinkelbach's algorithm 
	was established in~\cite{schaible1976fractional}.
	We derive a similar convergence result for Algorithm~\ref{alg:dinkelbach}.

    \begin{theorem}\label{thm:convdinkelbach}
		Algorithm \ref{alg:dinkelbach} terminates with the optimal value after
		finitely many steps, the rate of convergence is superlinear. 
	\end{theorem}
	To prove the convergence rate of \Cref{thm:convdinkelbach},
	we first need the following two lemmas.
	\begin{lemma}
		\label{lemma:ineq-dinkelbach}
		Let $\gamma'$, $\gamma''\in \real$ and $\vec x'$, $\vec x''\in \mathcal F$ be the optimal
		solutions of $P(\gamma')$ and $P(\gamma'')$, then
		$$ f(\vec x') - f(\vec x'') \leq
		\big( P(\gamma'') - (\gamma' - \gamma'') \vec e^\top \vec x'' \big)
		\Big( \frac{1}{\vec e^\top \vec x'} - \frac{1}{\vec e^\top \vec x''} \Big) .$$
	\end{lemma}
	\begin{proof}
		By the optimality of $\vec x'$ for $P(\gamma')$ it holds that
		$$
		\vec x'^\top \mat L \vec x' - \gamma' \vec e^\top \vec x' \leq \vec x''^\top \mat L \vec x'' - \gamma' \vec e^\top \vec x''.
		$$
		Dividing both sides by $\vec e^\top \vec x''$ and rearranging yields
		$$
		f(\vec x') \leq \frac{\vec x''^\top \mat L \vec x''}{\vec e^\top \vec x'} + \gamma' \Big( 1 - \frac{\vec e^\top \vec x''}{\vec e^\top \vec x'} \Big).
		$$
		Hence, we get that
		\begin{align*}
		f(\vec x') - f(\vec x'') &\leq \frac{\vec x''^\top \mat L \vec x''}{\vec e^\top \vec x'}
		+ \gamma' \Big( 1 - \frac{\vec e^\top \vec x''}{\vec e^\top \vec x'} \Big)
		- \frac{\vec x''^\top \mat L \vec x''}{\vec e^\top \vec x''}\\
		&= (\vec x''^\top \mat L \vec x'' -  \gamma' \vec e^\top \vec x'') \Big( \frac{1}{\vec e^\top \vec x'} - \frac{1}{\vec e^\top \vec x''} \Big)\\
		&=  (\vec x''^\top \mat L \vec x'' - \gamma'' \vec e^\top \vec x'' + \gamma'' \vec e^\top \vec x'' -  \gamma'\vec e^\top \vec x'') \Big( \frac{1}{\vec e^\top \vec x'} - \frac{1}{\vec e^\top \vec x''} \Big)\\
		&= \big( P(\gamma'') - (\gamma' - \gamma'') \vec e^\top \vec x'' \big)
		\Big( \frac{1}{\vec e^\top \vec x'} - \frac{1}{\vec e^\top \vec x''} \Big).
		\end{align*} 
	\end{proof}
	
	\begin{lemma}
		\label{lemma:denominator-dinkelbach}
		Let $\vec x'$ and $\vec x''$ be optimal solutions of $P(\gamma')$ and $P(\gamma'')$,
		then for $\gamma'' < \gamma'$ it holds that $\vec e^\top \vec x'' \leq \vec e^\top \vec x'$.
	\end{lemma}
	\begin{proof}
		From the optimality of $\vec x'$ and $\vec x''$ it follows that
		\begin{align*}
		\vec x'^\top \mat L \vec x' - \gamma' \vec e^\top \vec x' &\leq \vec x''^\top \mat L \vec x'' - \gamma' \vec e^\top \vec x'' \text{\quad and}\\
		\vec x''^\top \mat L \vec x'' - \gamma'' \vec e^\top \vec x'' &\leq \vec x'^\top \mat L \vec x' - \gamma'' \vec e^\top \vec x'.
		\end{align*}
		Adding the above two inequalities yields
		$$ (\gamma'' - \gamma') \vec e^\top \vec x' \leq (\gamma'' - \gamma') \vec e^\top \vec x''$$
		and hence the above claim holds. 
	\end{proof}

	\begin{proof}[Proof of \Cref{thm:convdinkelbach}]
                Let $\vec x^* \in \mathcal F$ be the optimum of the edge expansion problem, 
		i.e., $f(\vec x^*) = h(G)$ and let $\gamma^* = f(\vec{x}^*)$.
		From~\Cref{prop:dinkelbach-sign} we know that $P$ is a
		strictly decreasing piecewise linear function and
		therefore the algorithm terminates after finitely many
		iterations with	value $\gamma^*$.

		Let further $\gamma_i$ be the upper bound on $h(G)$ to check
		in the~$i$-th iteration of	Algorithm~\ref{alg:dinkelbach}.
		From \Cref{lemma:ineq-dinkelbach}, we get that 
		$$ \gamma_{i+1} - \gamma^* \leq (\gamma_i - \gamma^*)\Big(1 - \frac{\vec e^\top \vec x^*}{\vec e^\top \vec x_i} \Big) $$
		holds, since $P(\gamma^*) = 0$.
                The sequence  
		$$ \Big(1 - \frac{\vec e^\top \vec x^*}{\vec e^\top \vec x_i} \Big) $$
		is strictly decreasing (and converging to~0) as proved 
		in~\Cref{lemma:denominator-dinkelbach}.
		Therefore, the convergence rate of Algorithm~\ref{alg:dinkelbach}
		is superlinear.
	\end{proof}

	\subsection{Solving the parametrized optimization problem}\label{sec:solve_par}
	
	Evaluating $P(\gamma)$ requires solving a binary
	quadratic problem with two linear inequality constraints which is in general NP-hard.
	
	Most solvers for binary quadratic programs benefit from input data given as 
	integer as this aids the performance of the underlying branch-and-bound algorithm.
	Since we only consider rational values for $\gamma$, we 
	introduce the following parametric optimization problem
	\begin{equation*}
	Q(\gamma) = \min_{\vec x} \{ \gamma_d \vec x^\top \mat L \vec x - \gamma_n \vec e^\top \vec x: \vec x \in \F\}
	\end{equation*}
	for $\gamma = \nicefrac{\gamma_n}{\gamma_d}$ with integers $\gamma_{n} \ge 0$ and $\gamma_d > 0$.
	Observe that $Q(\gamma)=\gamma_d P(\gamma)$ and
	all considerations from above apply to this new formulation as well.
	
	Two performant parallel state-of-the-art solvers for binary quadratic programs with linear equality constraints are
	BiqBin~\cite{biqbin} and BiqCrunch~\cite{krislock2017biqcrunch}.
	BiqBin first transforms the problem into a QUBO in a pre-processing phase and then solves the equivalent max-cut problem
	with its max-cut solver. To solve $Q(\gamma)$, we make use of the problem specific properties and
	directly transform it into a max-cut problem ourselves.
	The first step towards achieving this is to obtain an exact formulation as a QUBO using binary encoding of the slack variables and the penalty parameter suggested in~\cite[Thm. 15]{expedis}. To aid our derivation, let us denote the integer $n_{s}$ and vector~$\vec v^{n_s} \in \real^{n_s+1}$ by
	\[
	n_s = \Big\lceil \log_2 \Big\lfloor \frac n 2 \Big\rfloor \Big\rceil - 1, \qquad  v^{n_s}_i = 2^{i-1} \text{ for all } i.
	\]

	\begin{proposition}
		Let $\mathcal F = \{\vec x \in \{0,1\}^n : 1 \leq \vec e^\top \vec x \leq \frac{n}{2}\}$
		and $\vec v^{n_s} \in \real^{n_s+1}$ with $ v^{n_s}_i = 2^{i-1}$ and
		$n_s = \lceil \log_2(\lfloor \frac n 2 \rfloor) \rceil - 1$.
		Then
		$$\mathcal F = \Big\{\vec x \in \{0,1\}^n :
		\vec e^\top \vec x - \vec \alpha^\top \vec v^{n_s} = 1,\ 
		\vec e^\top \vec x + \vec \beta^\top \vec v^{n_s} = \Big\lfloor \frac n 2  \Big\rfloor ,\ 
		\vec \alpha, \vec \beta \in \{0,1\}^{n_s + 1}
		\Big\} .$$
	\end{proposition}
	
	\begin{proof}
		For any $x \in \mathcal F$ it holds that $\vec e^\top \vec x = 1 + s = \lfloor \frac n 2 \rfloor -t $ for some slack variables $s$ and $t$ with $0 \leq s, t \leq \lfloor \frac n 2 \rfloor - 1$.
		In fact, any upper bound on $s$ and~$t$ greater or equal than
		$\lfloor \frac n 2 \rfloor - 1$ is fine, since from $\vec e^\top \vec x = 1 + s$ and 
		$s \geq 0$ it follows that $\vec e^\top \vec x \geq 1$ and
		from~$\vec e^\top \vec x = \lfloor \frac n 2 \rfloor - t$ and $t \geq 0$ it follows that
		$\vec e^\top \vec x \leq \lfloor \frac n 2 \rfloor$. 
		The smallest possible value for~$n_s$ is
		$\lceil \log_2(\lfloor \frac n 2 \rfloor) \rceil - 1$, since this gives an upper
		bound of $2^{n_s + 1} - 1$ on $s$ and $t$.
	\end{proof}

	\begin{proposition}
		Let $\vec x' \in \mathcal{F}$ with $\frac{\vec x'^\top \mat L \vec x'}{\vec e^\top \vec x'} = \frac{\gamma_n}{\gamma_d} = \gamma$ and $\gamma_d > 0$.
		The problem~$Q(\gamma)$ can then be equivalently formulated as the following QUBO,
		\begin{equation}\label{eq:dinkelbach-unconstrained}
		\begin{multlined}
		\min_{\vec x, \vec \alpha, \vec \beta} 
		\Bigg\{  \gamma_d \vec x^\top \mat L \vec x - \gamma_n \vec e^\top \vec x +
		\sigma \bigg\lVert
		\bigg(
		\begin{array}{c}
		\vec e^\top \vec x - \vec \alpha^\top \vec v^{n_s} -1 \\
		\vec e^\top \vec x  + \vec \beta^\top \vec v^{n_s} - \lfloor \frac n 2 \rfloor
		\end{array}
		\bigg)
		\bigg\rVert^2:\qquad \\
		\vec x \in \{0,1\}^n,\, \vec \alpha, \vec \beta \in \{0,1\}^{n_s +1}  \Bigg\}	
		\end{multlined}
		\end{equation}
		with $\sigma > \gamma_n n $.
	\end{proposition}
	
	\begin{proof}
		Let $g(\vec x, \vec \alpha, \vec \beta) $ denote the objective function of~\eqref{eq:dinkelbach-unconstrained}.
		For any feasible vector $\vec x \in \mathcal{F}$ there exist uniquely defined $\vec{\alpha_x}$, $\vec{\beta_x}$
		such that $g(\vec x, \vec{\alpha_x}, \vec{\beta_x}) = \gamma_d \vec x^\top \mat L \vec x - \gamma_n \vec e^\top \vec x$. Thus, the objective function of~$Q(\gamma)$ and~\eqref{eq:dinkelbach-unconstrained} coincide
		for~$\vec x \in \mathcal{F}$.
		Moreover, for $\vec x'$ it holds that~$g(\vec x', \vec{\alpha_{x'}}, \vec{\beta_ {x'}}) = 0$.
		
		For $\vec x \in \{0,1\}^n \setminus \mathcal F$ there do not
		exist $\vec \alpha$, $\vec \beta \in \{0,1\}^{n_s + 1}$
		such that both equalities
		$\vec e^\top \vec x - \vec \alpha^\top \vec v^{n_s} = 1$ and
		$\vec e^\top \vec x + \vec \beta^\top \vec v^{n_s} = \big\lfloor \frac n 2  \big\rfloor$
		are satisfied, as one of the slack variables has to be negative in order to fulfill the
		constraints.
		Additionally, since $\mat L$ is positive semidefinite we can conclude that
		$$g(\vec x, \vec \alpha, \vec \beta) \geq -\gamma_n n + \sigma > 0
		= g(\vec x', \vec{\alpha_{x'}}, \vec{\beta_ {x'}}).$$
		Therefore, $Q(\gamma)$ is an equivalent formulation of~\eqref{eq:dinkelbach-unconstrained}.
	\end{proof}
	
	The unconstrained binary quadratic program~\eqref{eq:dinkelbach-unconstrained} can again be transformed to a max-cut problem,
	as explained in~\cite{Barahona1989} for example. Applied to our
	problem we obtain the following result.
	
	\begin{corollary}
		Let $G=(V,E)$ and let $G''$ be the graph with vertices from $V$ plus
		the vertices $v_0, v_{\alpha_0}, \dots, v_{\alpha_{n_s}},
		v_{\beta_0}, \dots, v_{\beta_{n_s}}$ for the variable
		vectors~$\vec \alpha$ and~$\vec \beta$.
		
		Let the weights $c_{uw}$ on the edges of $G''$ be as follows.
		\begin{equation*}
		c_{uv_0} =\begin{cases}
		2\sigma(n - 1 - \lfloor \frac n 2 \rfloor) - \gamma_n & \text{ if } u \in V(G)\\
		2\sigma(2^{n_s} - \frac{n-1}{2})2^i & \text{ if } u= v_{\alpha_i}\\
		2\sigma(2^{n_s} - \lfloor \frac n 2 \rfloor + \frac{n - 1}{2})2^i
		& \text{ if } u = v_{\beta_i}\end{cases}
		\end{equation*}
		\begin{equation*}
		c_{uw} = \begin{cases}
		-2^i\sigma & \text{ if } u = v_{\alpha_i} \text{ and } w \in V(G)\\
		2^i\sigma & \text{ if } u = v_{\beta_i} \text{ and } w \in V(G)\\
		2^{i+j}\sigma & \text{ if } u = v_{\alpha_i} \text{ and } w =
		v_{\alpha_j}\\
		2^{i+j}\sigma & \text{ if } u = v_{\beta_i} \text{ and } w = v_{\beta_j}\\
		\end{cases}
		\end{equation*}
		For $u\in V(G)$ and $w \in V(G)$, we have 
		\begin{equation*}
		c_{uw} = \begin{cases}
		2\sigma - \gamma_d & \text{ if } uw \in E(G) \\
		2\sigma  & \text{ if } uw \not\in E(G) 
		\end{cases}
		\end{equation*}
		Edges not specified above have weight zero.
		Let the penalty parameter be~$\sigma = \gamma n + 1$. Then all weights
		are integers and it
		holds that $Q(\gamma) = \mathtt{offset} - \operatorname{max-cut}(G'')$
		where
		\begin{equation*}
		\begin{multlined}
		\mathtt{offset} =
		- \gamma_n  n
		+  \sigma \cdot\bigg[2^{n_s + 2}
		\Big( 2\cdot 2^{n_s} -\Big\lfloor \frac n 2 \Big \rfloor - 1\Big)
		+ 2n^2 - 2n + 1\\
		+ \Big\lfloor \frac n 2 \Big \rfloor \cdot \Big(\Big\lfloor \frac n 2 \Big \rfloor
		-2n + 2\Big) \bigg].
		\end{multlined}
		\end{equation*}
	\end{corollary}

	\section{Numerical results}
	\label{sec:numericalresults}
	All of our algorithms were written\footnote{The code is available on
		the arXiv page \url{https://arxiv.org/abs/2403.04657} and on \url{https://github.com/melaniesi/EdgeExpansion.jl}} in Julia~\cite{bezanson2017julia}
	version 1.9.2.
	That is, the split-and-bound Algorithm~\ref{alg:split-and-bound}
	including pre-elimination and 
	the transformation from $k$-bisection to max-cut problems.
	Also, Algorithm~\ref{alg:dinkelbach} we implemented in Julia.
	The SDPs to compute our cheap lower bounds~$\ell_k$ from the bisection problem in \cref{eq:cheaplb} are solved with
	MOSEK 10.0~\cite{mosek10.0julia} using JuMP~\cite{JuMP2023}.
        We also use JuMP to solve MIQCPs with Gurobi~\cite{gurobi} version 11.0.
	The solver BiqBin~\cite{biqbin} for binary quadratic problems was used to solve the parametrized problems in Dinkelbach's method, and we extended the C code of this solver by adding the option
	to provide an initial lower bound on the maximization problem. 
	The corresponding changes are tracked in the git repository \url{https://gitlab.aau.at/BiqBin/biqbin}.
	All computations were carried out on an AMD EPYC 7532 with 32~cores with
	3.30GHz and 1024GB~RAM, operated under Debian GNU/Linux~11.

	\subsection{Benchmark instances}
	%In this paper, we consider the following benchmark instances.
	\paragraph{Randomly generated 0/1-polytopes}
	The first class of graphs are the graphs of random 0/1-polytopes.
	The polytopes are generated by randomly selecting $n_d$~vertices of the polytope in dimension~$d$, i.e., $n_d$~different
	0/1-vectors in dimension~$d$. To obtain the graph, we then
	solve a linear programming feasibility problem to check whether
	there is an edge for a given pair of vertices.
    For any pair~$(d,n_d)$ with $d\in\{8,9,10\}$ and $n_8 \in \{164, 189\}$, $n_9 \in \{153,178,203,228,253,278\}$, and
	$n_ {10} \in \{256, 281\}$, we generated~3 random~0/1-polytopes.
		The choice of these values is motivated by the aim to generate graphs with~150
		to~300~vertices, and to randomly sample between~25\% and~75\% of the possible 0/1-vectors,
		with percentages chosen in incremental steps within this range.

	\paragraph{Grlex and grevlex graphs}
	Another class of graphs we consider are the graphs of grlex and grevlex
	polytopes introduced and characterised by Gupte and Poznanovi\'{c}~\cite{gupte2019dantzigfigures}.
	The grevlex-$d$ and grlex-$d$ instances of our benchmark set
	are the corresponding graphs of the polytopes in dimension~$d$.
	It was shown in~\cite{gupte2019dantzigfigures} that those graphs have
	a very specific structure and that the edge expansion of all
	grlex-$d$ instances is~1.
	
	\paragraph{DIMACS and Network graphs}
	The last category of graphs originates from the graph partitioning
	and clustering application.
	While the previous benchmark instances are graphs of polytopes, these instances
	model relations and networks.
	The set of DIMACS instances are the graphs of the 10th DIMACS challenge on graph partitioning and graph clustering~\cite{Bader2013}
	with at most 500~vertices.
	Additionally, we consider some more network graphs obtained from the online
	network repository~\cite{netzschleuder}.
	We in particular chose connected graphs with structural properties such as multiple clusters %malaria genes
	or having some vertices of high degree and several vertices with small degree.

	\subsection{Discussion of the experiments}
	
	We compare different algorithms for computing the edge expansion of a graph, namely
	\begin{enumerate}
		\item Split \& bound Algorithm~\ref{alg:split-and-bound}, 
		\item Fractional programming using Discrete Newton-Dinkelbach's method in Algorithm~\ref{alg:dinkelbach},
                \item Gurobi for solving the MIQCP.
	\end{enumerate}

	\paragraph{Algorithm~\ref{alg:split-and-bound} vs. Algorithm~\ref{alg:dinkelbach} vs. Gurobi}
	The detailed results of our experiments are given
	in \Cref{tab:result-rand01,tab:result-grlex,tab:result-grevlex,tab:result-dimacs,tab:result-network}. 
	In each of the tables, the first column gives the name of the
	instance followed by the number of vertices and edges. Column~4
	reports the optimal solution, i.e., the edge expansion of the graph.
	
	In the split \& bound section of the table, the first two columns
	give the global lower and
	upper bound after the pre-elimination Algorithm~\ref{alg:pre-elimination}. 
	The number of candidates for~$k$ after the pre-elimination is given in
	column~3.
	In column~4 we report the number of indices~$k\in \mathcal I$ we were able to
	eliminate after solving the root node of the branch-and-bound tree.
	Column~5 lists the total number of branch-and-bound nodes in the max-cut algorithm for all
	values of $k$ considered. 
	The last two columns display the time spent in the pre-elimination
	and the total time (including pre-elimination) of the algorithm. 
	
	In the section for Dinkelbach's algorithm, the first column 
	gives the first guess for the edge expansion, i.e., the first
	trial for $\gamma$. As described before, we take the upper bound from
	the heuristic for this initialization. Note, that this first guess may differ
	from~$u^*$, since in the pre-elimination step of split \& bound we perform
	30 additional rounds of simulated annealing for all indices~$k \in \mathcal{I}$.
	Column~2 indicates how many parametrized problems $P(\gamma_i)$ have been
	solved, and column~3 gives the total number of branch-and-bound nodes
	for solving all parametrized problems.
	The fourth column of the results of Dinkelbach's algorithm displays
	the total time, including running the heuristic to obtain the
	first guess.

        The final column of the tables holds information about
        computing the edge expansion using Gurobi. For the graphs from
        the randomly generated polytopes, Gurobi did not succeed to
        solve any of these instances within a time limit of 3~hours,
        we therefore report the gap after this time limit
        in \Cref{tab:result-rand01}. 
        In \Cref{tab:result-grlex,tab:result-grevlex,tab:result-dimacs,tab:result-network}
        we report the time for computing the edge expansion.

        We highlight in the tables, which of our two algorithms
        performs better.

	\paragraph{Algorithm~\ref{alg:split-and-bound} (Split \& bound)}
	As can be seen in all tables, the pre-elimination phase of
	split \& bound only leaves a
	comparably small number of candidates for $k$ to be further
	investigated. Remember that the number of potential candidates is
	$\lfloor \frac n 2 \rfloor$, whereas $|\mathcal{I}|$ is the number of
	candidates that remain after the pre-elimination.
	For~12~instances we were able to compute the edge expansion even
	within the pre-elimination phase, i.e. the set of candidates~$\mathcal I$ was empty.
	For the randomly generated 0/1-polytopes on average only 12\% of the
	candidates have to be further examined, for the other instance classes
	we can approximately eliminate 80\% of the candidate values for~$k$ on average.
	Only for~6 instances we were not able to halve the number of candidates within
	the pre-elimination phase. All of those~6~instances are from the DIMACS or network graphs
	test set.
	This indicates that in general already the cheap SDP bound is of good quality.
	
	We also observe that the SDP bound in the root-node of the
	branch-and-bound tree is of high quality: in 48 out of the 67
	instances the gap is closed within the root node for all candidates to
	be considered.
	For the other instances the average percentage of candidates
	left after the root is~8\%, where only for one instance 
	the number of remaining candidates is still above~20\%.
	
	The heuristic for computing upper bounds also performs extremely well:
	for almost all instances the upper bound found is the edge expansion
	of the graph, see columns titled $h(G)$ and $u^*$. In fact,
	only for 3~instances the heuristic fails to find the optimal solution.
	
	Overall, we solve almost all of the considered instances within a few
	minutes, for very few instances the branch-and-bound tree grows rather
	large and therefore computation times exceed several hours.
	
	\paragraph{Algorithm~\ref{alg:dinkelbach} (Discrete Newton-Dinkelbach)}
	Whenever the heuristic already returns the value of the optimal
	solution, we only have to solve one parametrized problem to certify
	the optimality of this value. For most of the instances tested, this
	certificate is already obtained in the root node of the
	branch-and-bound tree. However, there are many instances among the random 0/1 polytopes where
	BiqBin terminates because of numerical problems even for the
	first parametrized problem, see \Cref{tab:result-rand01}. This in particular
	arises when  $\gamma_d$ and $\gamma_n$ (see Section \cref{sec:solve_par}) are large.
	
        \paragraph{Solving the MIQCP with Gurobi}
        To compute the edge expansion using Gurobi, we input the last
        formulation of~\eqref{hforms} adding the redundant constraint
        $y \ge 0$. Without this constraint, Gurobi terminated only
        after 1.65~hours/3\,548 work units (resp. more than
        24~hours/59\,000 work units) on a graph with 29~vertices and
        119~edges (resp. 37~vertices and 176~edges) corresponding to
        the grevlex polytope in dimension~7 (resp. 8).
        
        Adding the redundant constraint, Gurobi is very efficient for
        graphs with an expansion less than one,
        see \Cref{tab:result-dimacs,tab:result-network}, but as soon
        as the expansion (and also the number of vertices of the
        graph) gets larger, Gurobi cannot solve the instance within a
        few hours, 
        see \Cref{tab:result-rand01,tab:result-grevlex}.  

        \paragraph{Impact of the edge expansion}
          The performance of an
          algorithm is not in particular depending on their size or
          density.   
        As noted above, Gurobi is very efficient on graphs with an
        edge expansion less than one.
        As the expansion gets larger, our algorithms are the clear
        winners over Gurobi.

        As for the performance of Algorithm~\ref{alg:dinkelbach},
        we observe
        that for large expansion the algorithm performs very well in
        general. But we observe, that in particular if the edge expansion
        is a fraction of large (coprime) nominator and denominator,
        this is a disadvantage of the algorithm.

        For split \& bound, there seems to be no impact of the edge
        expansion on the performance.

	\paragraph{Conclusion}
        To summarize the results, we give a performance profile
        in \Cref{fig:performance}. Gurobi solves the MIQCP very
        efficiently for several instances, but fails to yield results
        for others within a time limit of 3~hours. It is the clear
        winner for instances with very small edge expansion.
	Comparing split \& bound with the algorithm following the Discrete
	Newton-Dinkelbach method,
	we observe the following behavior. 
	For the grlex instances, Dinkelbach performs extremely well compared to the
	split \& bound approach, see \Cref{tab:result-grlex}.
	Whereas for the grevlex instances in \Cref{tab:result-grevlex}, we observe that, except for dimension~13, the split \& bound algorithm by
	far outperforms Algorithm~\ref{alg:dinkelbach}.
	In addition to the already mentioned, there are some other instances where
	the difference in the runtimes between the two algorithms is significant.
	For example, on the instances \texttt{rand01-9-153-0} and \texttt{malaria\_genes\_HVR1}
	split \& bound clearly dominates Algorithm~\ref{alg:dinkelbach}, whereas the latter is significantly better on the instances \texttt{rand01-9-2781} and \texttt{sp-office}.
	
	The conclusion is that in general for graphs with larger edge expansion, the split~\& bound algorithm is best, and for graphs with small
	edge expansion Algorithm~\ref{alg:dinkelbach} has a better
	performance than split~\& bound, but there are a few exceptions, and the difference in the
	total time of solving an instance can be quite large.

        Overall, we conclude that with our algorithms we can compute the edge
        expansion of various graphs of size up to around 400 vertices
        and no other algorithm can achieve this. The time for solving
        an instance varies, it can be a few seconds for very
        structured instances and it can exceed one hour, in particular
        for the instances coming from 0/1-polytopes with rather large
        expansion. For standard branch-and-cut solvers like Gurobi
        these instances are out of reach.

	\begin{sidewaystable}
		\begin{center}
			\tiny
			\begin{tabular}{l|r|r|r||r|r|r|r|r|r|r||r|r|r|r||l}
				\multicolumn{4}{@{}c@{}}{}
				& \multicolumn{7}{@{}c@{}}{Algorithm~\ref{alg:split-and-bound} (split \& bound)}
				& \multicolumn{4}{@{}c@{}}{Algorithm~\ref{alg:dinkelbach} (Dinkelbach)}
				& \multicolumn{1}{@{}c@{}}{Gurobi}
				\\\cmidrule(r){1-4}\cmidrule(lr){5-11}\cmidrule(lr){12-15}\cmidrule(l){16-16}%
				\head{Instance} & \multicolumn{1}{@{}c@{}|}{$n$} &\multicolumn{1}{@{}c@{}|}{$m$} & \multicolumn{1}{@{}c@{}||}{ $h(G)$ } & \multicolumn{1}{@{}c@{}|}{$\min\, \ell_k$} & \multicolumn{1}{@{}c@{}}{$u^*$}& \multicolumn{1}{|@{}c@{}|}{$\lvert \mathcal I\rvert$} & \head{solved\\in root}& \head{B\&B\\nodes} &  \head{Alg.~\ref{alg:pre-elimination}\\time (s)} & \head{total\\time (s)} & \multicolumn{1}{@{}c@{}|}{\head{first\\guess}} & \head{\# of\\steps} & \head{B\&B\\nodes} & \head{total\\time (s)} & \head{relative\\gap} \\ \midrule
				\begin{filecontents*}{summary-splitnbound-rand01.csv}
instance,n,m,lambdatwohalf,lb,ub,opt,sizeI,bbnodesTotal,timePreelim,timeTotal,morebbnodes,solvedinroot,firstguess,dinkelbachsteps,dinkelbachBbnodestotal,dinkelbachtimetotal,gurobiRelGap,gurobiSolTime,timeSPLIThighlighted,timeDBhighlighted,timeSPLIThighlighted2,timeDBhighlighted2,timeGUROBIhighlighted2
%rand01-7-32-0,32,308,5.4966,7.6875,7.9375,7.9375,1,1,0.5,0.5,0,1,7.9375,1,1,1.7,-,-,\bftab{0.5},1.7,\bftab{0.5},1.7,-
%rand01-7-32-1,32,319,6.5032,7.8750,8.0000,8.0000,1,1,0.5,0.6,0,1,8.0000,1,1,0.3,-,-,0.6,\bftab{0.3},0.6,\bftab{0.3},-
%rand01-7-32-2,32,310,5.4796,7.6875,7.8750,7.8750,1,1,0.8,0.9,0,1,7.8750,1,1,1.4,-,-,\bftab{0.9},1.4,\bftab{0.9},1.4,-
%rand01-7-57-0,57,593,4.5981,6.7143,7.1786,7.1786,3,3,2.4,8.6,0,3,7.2143,2,2,18.9,-,-,\bftab{8.6},18.9,\bftab{8.6},18.9,-
%rand01-7-57-1,57,555,3.1223,5.8214,6.3571,6.3571,4,4,2.5,4.7,0,4,6.3571,1,1,15.3,-,-,\bftab{4.7},15.3,\bftab{4.7},15.3,-
%rand01-7-57-2,57,598,4.3905,6.7857,7.3571,7.3571,4,4,3.0,12.1,0,4,7.4643,2,2,33.5,-,-,\bftab{12.1},33.5,\bftab{12.1},33.5,-
%rand01-11-60-0,60,1676,25.5546,27.1667,27.2667,27.2667,1,1,3.1,3.6,0,1,27.2667,1,1741,6437.9,-,-,\bftab{3.6},6437.9,\bftab{3.6},6437.9,-
%rand01-8-64-0,64,1116,10.7147,13.4063,14.0625,14.0625,2,2,3.7,13.0,0,2,14.0938,2,26,189.8,-,-,\bftab{13.0},189.8,\bftab{13.0},189.8,-
%rand01-8-64-1,64,1043,9.0508,12.0000,12.6563,12.6563,2,2,3.7,10.1,0,2,12.6563,1,7,81.6,-,-,\bftab{10.1},81.6,\bftab{10.1},81.6,-
%rand01-8-64-2,64,1103,8.5655,13.0938,13.5313,13.5313,2,2,3.6,4.9,0,2,13.5313,1,7,60.8,-,-,\bftab{4.9},60.8,\bftab{4.9},60.8,-
%rand01-11-65-0,65,1947,27.4260,29.3438,29.4688,29.4688,1,1,4.4,5.4,0,1,29.4688,1,403,1485.3,-,-,\bftab{5.4},1485.3,\bftab{5.4},1485.3,-
%rand01-11-70-0,70,2233,28.4213,30.4571,30.6286,30.6286,1,1,5.6,7.4,0,1,30.6286,1,179,1571.9,-,-,\bftab{7.4},1571.9,\bftab{7.4},1571.9,-
%rand01-7-82-0,82,692,2.8942,4.0976,4.4878,4.4878,5,5,6.8,38.7,0,5,4.4878,1,11,242.4,-,-,\bftab{38.7},242.4,\bftab{38.7},242.4,-
%rand01-7-82-1,82,646,2.5367,3.5366,4.0976,4.0976,8,8,7.8,76.9,0,8,4.0976,1,3,81.2,-,-,\bftab{76.9},81.2,\bftab{76.9},81.2,-
%rand01-7-82-2,82,673,2.6294,3.9268,4.5610,4.5366,9,9,8.4,66.3,0,9,4.6000,2,96,700.1,-,-,\bftab{66.3},700.1,\bftab{66.3},700.1,-
%rand01-8-89-0,89,1503,8.1266,11.4091,12.0909,12.0682,4,8,8.7,67.3,1,3,12.1818,2,52,588.9,-,-,\bftab{67.3},588.9,\bftab{67.3},588.9,-
%rand01-8-89-1,89,1469,7.7708,10.7955,11.6818,11.6818,5,11,9.1,121.7,1,4,11.6818,1,11,164.1,-,-,\bftab{121.7},164.1,\bftab{121.7},164.1,-
%rand01-8-89-2,89,1519,8.4334,11.5455,12.2500,12.2500,4,4,9.3,25.2,0,4,12.2500,1,11,182.3,-,-,\bftab{25.2},182.3,\bftab{25.2},182.3,-
%rand01-8-114-0,114,1771,4.9775,9.0877,9.8947,9.8947,7,7,18.5,129.6,0,7,9.8947,1,15,315.6,-,-,\bftab{129.6},315.6,\bftab{129.6},315.6,-
%rand01-8-114-1,114,1667,4.8377,8.2281,9.1228,9.1228,7,7,17.8,122.7,0,7,9.1228,1,87,1434.1,-,-,\bftab{122.7},1434.1,\bftab{122.7},1434.1,-
%rand01-8-114-2,114,1767,6.3503,9.0000,9.8947,9.8947,7,13,17.1,186.3,1,6,9.8947,1,1,69.9,-,-,186.3,\bftab{69.9},186.3,\bftab{69.9},-
%rand01-9-128-0,128,3506,15.3003,19.4844,20.2813,20.2813,3,3,26.4,40.8,0,3,20.2813,1,3,188.9,-,-,\bftab{40.8},188.9,\bftab{40.8},188.9,-
%rand01-9-128-1,128,3446,14.1729,19.0469,19.9531,19.9531,4,4,26.9,53.4,0,4,19.9531,-,-,-,-,-,\bftab{53.4},-,\bftab{53.4},-,-
%rand01-9-128-2,128,3474,14.6494,19.1719,20.2188,20.2188,4,4,26.6,72.4,0,4,20.2188,1,1,42.8,-,-,72.4,\bftab{42.8},72.4,\bftab{42.8},-
%rand01-8-139-0,139,1902,4.6525,7.0145,7.9130,7.8986,11,33,34.9,539.0,2,9,7.9130,2,12,622.7,-,-,\bftab{539.0},622.7,\bftab{539.0},622.7,-
%rand01-8-139-1,139,1782,3.5336,6.2754,7.2029,7.2029,14,18,36.3,476.4,1,13,7.2029,1,93,1693.6,-,-,\bftab{476.4},1693.6,\bftab{476.4},1693.6,-
%rand01-8-139-2,139,1880,5.0488,6.8261,7.4348,7.4348,8,8,33.2,216.0,0,8,7.4348,1,1,51.5,-,-,216,\bftab{51.5},216,\bftab{51.5},-
rand01-9-153-0,153,4081,13.8294,17.7763,18.7500,18.7500,5,5,43.2,129.4,0,5,18.7500,1,147,3397.6,0.900,10800,\bftab{129.4},3397.6,\bftab{129.4},3397.6,10800
rand01-9-153-1,153,4044,13.8721,17.5789,18.4868,18.4868,5,5,39.9,111.9,0,5,18.4868,-,-,-,0.898,10800,\bftab{111.9},-,\bftab{111.9},-,10800
rand01-9-153-2,153,4107,14.0106,17.8421,19.0000,19.0000,6,6,45.2,220.4,0,6,19.0000,1,83,1371.1,0.899,10800,\bftab{220.4},1371.1,\bftab{220.4},1371.1,10800
rand01-8-164-0,164,1868,3.5594,4.8659,5.7683,5.7683,17,123,62.0,2037.7,6,11,5.7683,1,27,1610.4,0.809,10800,2037.7,\bftab{1610.4},2037.7,\bftab{1610.4},10800
rand01-8-164-1,164,1837,3.2212,4.7073,5.3537,5.3537,15,27,56.9,774.7,3,12,5.3537,1,61,1786.1,0.785,10800,\bftab{774.7},1786.1,\bftab{774.7},1786.1,10800
rand01-8-164-2,164,1808,2.7731,4.7561,5.7439,5.7439,29,251,85.3,5347.0,24,5,5.7561,2,78,2743.5,0.833,10800,5347,\bftab{2743.5},5347,\bftab{2743.5},10800
rand01-9-178-0,178,4590,12.9479,16.0899,17.0787,17.0787,6,18,92.6,320.4,2,4,17.0787,-,-,-,0.919,10800,\bftab{320.4},-,\bftab{320.4},-,10800
rand01-9-178-1,178,4467,11.1118,15.3933,16.7079,16.7079,9,11,87.9,506.8,1,8,16.7079,-,-,-,0.899,10800,\bftab{506.8},-,\bftab{506.8},-,10800
rand01-9-178-2,178,4537,10.4321,15.6517,16.7528,16.7528,7,7,70.0,219.1,0,7,16.7528,-,-,-,0.920,10800,\bftab{219.1},-,\bftab{219.1},-,10800
rand01-8-189-0,189,1768,2.5198,3.4681,4.2234,4.2234,23,633,99.2,5581.6,12,11,4.2340,2,14,1036.0,0.817,10800,5581.6,\bftab{1036.0},5581.6,\bftab{1036.0},10800
rand01-8-189-1,189,1745,2.4291,3.3723,4.0426,4.0426,26,128,103.8,2634.7,6,20,4.0426,1,33,1603.1,0.795,10800,2634.7,\bftab{1603.1},2634.7,\bftab{1603.1},10800
rand01-8-189-2,189,1719,2.2431,3.3511,4.0745,4.0638,28,100,97.9,2669.6,9,19,4.0851,2,38,2014.7,0.788,10800,2669.6,\bftab{2014.7},2669.6,\bftab{2014.7},10800
rand01-9-203-0,203,4900,10.8977,14.0198,15.1386,15.1386,9,41,109.7,892.1,5,4,15.1386,-,-,-,0.930,10800,\bftab{892.1},-,\bftab{892.1},-,10800
rand01-9-203-1,203,4781,9.7351,13.5545,14.8416,14.8416,12,388,117.2,3591.5,10,2,14.8515,-,-,-,0.925,10800,\bftab{3591.5},-,\bftab{3591.5},-,10800
rand01-9-203-2,203,4720,8.5527,13.3861,14.3762,14.3762,9,9,105.8,412.1,0,9,14.3762,-,-,-,0.945,10800,\bftab{412.1},-,\bftab{412.1},-,10800
rand01-9-228-0,228,5065,8.6929,12.0439,13.2368,13.2368,13,129,166.0,2083.8,6,7,13.2368,-,-,-,0.943,10800,\bftab{2083.8},-,\bftab{2083.8},-,10800
rand01-9-228-1,228,4927,4.2158,9.0000,9.0000,9.0000,0,0,135.6,135.6,0,0,9.0000,1,23,586.7,0.857,10800,\bftab{135.6},586.7,\bftab{135.6},586.7,10800
rand01-9-228-2,228,4984,6.9005,11.8070,12.8246,12.8246,11,11,174.3,619.9,0,11,12.8246,-,-,-,0.937,10800,\bftab{619.9},-,\bftab{619.9},-,10800
rand01-9-253-0,253,5258,7.0229,10.6825,11.8730,11.8730,16,684,234.5,10547.7,8,8,11.9760,-,-,-,0.955,10800,\bftab{10547.7},-,\bftab{10547.7},-,10800
rand01-9-253-1,253,5053,4.1818,9.0000,9.0000,9.0000,0,0,186.9,186.9,0,0,9.0000,1,1,121.3,0.910,10800,186.9,\bftab{121.3},186.9,\bftab{121.3},10800
rand01-9-253-2,253,5072,6.4856,10.1190,11.2222,11.2222,16,402,232.7,8709.2,13,3,11.2302,-,-,-,0.970,10800,\bftab{8709.2},-,\bftab{8709.2},-,10800
rand01-10-256-0,256,11056,19.8863,29.4219,30.4766,30.4766,5,5,228.8,547.7,0,5,30.4766,-,-,-,0.967,10800,\bftab{547.7},-,\bftab{547.7},-,10800
rand01-10-256-1,256,10611,17.0468,27.7031,28.8438,28.8438,6,18,233.5,926.9,3,3,28.8438,1,1,308.6,0.969,10800,926.9,\bftab{308.6},926.9,\bftab{308.6},10800
rand01-10-256-2,256,10746,21.9818,28.1563,29.3750,29.3750,6,20,240.6,769.7,3,3,29.3750,1,7,607.2,0.966,10800,769.7,\bftab{607.2},769.7,\bftab{607.2},10800
rand01-9-278-0,278,5224,5.8989,8.9065,10.0719,10.0719,20,1292,326.8,17542.8,17,3,10.0719,-,-,-,0.953*,10800,\bftab{17542.8},-,17542.8,-,10800
rand01-9-278-1,278,5007,3.9986,8.3237,9.0000,9.0000,15,387,336.6,8153.3,12,3,9.0000,1,1,103.7,0.954,10800,8153.3,\bftab{103.7},8153.3,\bftab{103.7},10800
rand01-9-278-2,278,5132,5.9765,8.6906,9.9209,9.9209,22,2238,338.1,31125.4,16,6,9.9209,-,-,-,0.957*,10800,\bftab{31125.4},-,31125.4,-,10800
rand01-10-281-0,281,11828,19.8000,27.7357,28.9000,28.9000,7,75,311.7,1807.9,4,3,28.9000,-,-,-,0.975,10800,\bftab{1807.9},-,\bftab{1807.9},-,10800
rand01-10-281-1,281,11490,14.8451,26.5214,27.7929,27.7929,8,30,321.2,1776.4,3,5,27.8071,-,-,-,0.973,10800,\bftab{1776.4},-,\bftab{1776.4},-,10800
rand01-10-281-2,281,11454,18.8491,26.4571,27.7500,27.7500,8,66,316.9,2435.7,4,4,27.7500,1,11,1103.5,0.972,10800,2435.7,\bftab{1103.5},2435.7,\bftab{1103.5},10800
\end{filecontents*}
\csvreader[head to column names, late after line = \\]{summary-splitnbound-rand01.csv}%
				{instance=\name, n=\size, m=\edges }{%
					\name & \size & \edges & \opt &\lb & \ub & \sizeI & \solvedinroot & \bbnodesTotal & \timePreelim & \timeSPLIThighlighted & \firstguess & \dinkelbachsteps &
					\dinkelbachBbnodestotal & \timeDBhighlighted & \gurobiRelGap
				} \bottomrule
			\end{tabular}
			\caption{Comparison of Algorithm~\ref{alg:split-and-bound}
				(split \& bound),
		Algorithm~\ref{alg:dinkelbach} (Dinkelbach) and Gurobi for graphs
		of random 0/1-polytopes. A ``-'' indicates that Algorithm~\ref{alg:dinkelbach} could not solve this instance.
                The last column displays the
		gap reported by Gurobi after a time limit of 3~hours.\\[0.2em]
		{\scriptsize * Increased timelimit of Gurobi to match the time one of the other algorithms needed for exact computation.}}
			\label{tab:result-rand01}
		\end{center}
	\end{sidewaystable}
	
	\begin{sidewaystable}
		\centering
		\scriptsize
		\begin{tabular}{l|r|r|r||r|r|r|r|r|r|r||r|r|r|r||r}
			\multicolumn{4}{@{}c@{}}{}
			& \multicolumn{7}{@{}c@{}}{Algorithm~\ref{alg:split-and-bound} (split \& bound)}
			& \multicolumn{4}{@{}c@{}}{Algorithm~\ref{alg:dinkelbach} (Dinkelbach)}
			& \multicolumn{1}{@{}c@{}}{Gurobi}
			\\\cmidrule(r){1-4}\cmidrule(lr){5-11}\cmidrule(lr){12-15}\cmidrule(l){16-16}%
			\head{Instance} & \multicolumn{1}{@{}c@{}|}{$n$} &\multicolumn{1}{@{}c@{}|}{$m$} & \multicolumn{1}{@{}c@{}||}{\,$h(G)$\,} & \multicolumn{1}{@{}c@{}|}{$\min\, \ell_k$} & \multicolumn{1}{@{}c@{}}{$u^*$}& \multicolumn{1}{|@{}c@{}|}{$\lvert \mathcal I\rvert$} & \head{solved\\in root}& \head{B\&B\\nodes} &  \head{Alg.~\ref{alg:pre-elimination}\\time (s)} & \head{total\\time (s)} & \multicolumn{1}{@{}c@{}|}{\,\head{first\\guess}\,} & \head{\# of\\steps} & \head{B\&B\\nodes} & \head{total\\time (s)} & \head{total\\time (s)}\\ \midrule
			\begin{filecontents*}{summary-splitnbound-grlex.csv}
instance,n,m,lambdatwohalf,lb,ub,opt,sizeI,bbnodesTotal,timePreelim,timeTotal,morebbnodes,solvedinroot,firstguess,dinkelbachsteps,dinkelbachBbnodestotal,dinkelbachtimetotal,gurobiRelGap,gurobiSolTime,timeSPLIThighlighted,timeDBhighlighted,timeSPLIThighlighted2,timeDBhighlighted2,timeGUROBIhighlighted2
%grlex-2,4,4,1.0000,1.0000,1,1,0,0,0.1,0.1,0,0,1,1,1,0.1,-,-,\bftab{0.1},\bftab{0.1},\bftab{0.1},\bftab{0.1},-
%grlex-3,7,11,0.7929,1.0000,1,1,0,0,0.1,0.1,0,0,1,1,1,0.1,-,-,\bftab{0.1},\bftab{0.1},\bftab{0.1},\bftab{0.1},-
%grlex-4,11,24,0.6662,1.0000,1,1,0,0,0.1,0.1,0,0,1,1,1,0.1,-,-,\bftab{0.1},\bftab{0.1},\bftab{0.1},\bftab{0.1},-
%grlex-5,16,45,0.5811,1.0000,1,1,0,0,0.1,0.1,0,0,1,1,1,0.1,-,-,\bftab{0.1},\bftab{0.1},\bftab{0.1},\bftab{0.1},-
%grlex-6,22,76,0.5231,1.0000,1,1,0,0,0.2,0.2,0,0,1,1,1,0.3,-,-,\bftab{0.2},0.3,\bftab{0.2},0.3,-
grlex-7,29,119,0.4820,1.0000,1,1,0,0,0.3,0.3,0,0,1,1,1,0.3,0,0.3,\bftab{0.3},\bftab{0.3},\bftab{0.3},\bftab{0.3},\bftab{0.3}
grlex-8,37,176,0.4516,1.0000,1,1,0,0,0.6,0.6,0,0,1,1,1,0.9,0,0.6,\bftab{0.6},0.9,\bftab{0.6},0.9,\bftab{0.6}
grlex-9,46,249,0.4285,1.0000,1,1,0,0,1.5,1.5,0,0,1,1,1,0.8,0,0.8,1.5,\bftab{0.8},1.5,\bftab{0.8},\bftab{0.8}
grlex-10,56,340,0.4103,0.8571,1,1,7,7,2.7,22.7,0,7,1,1,1,2.4,0,1.2,22.7,\bftab{2.4},22.7,2.4,\bftab{1.2}
grlex-11,67,451,0.3958,0.8333,1,1,12,12,3.6,148.0,0,12,1,1,1,4.4,0,2.0,148,\bftab{4.4},148,4.4,\bftab{2.0}
grlex-12,79,584,0.3839,0.8000,1,1,15,15,5.8,280.4,0,15,1,1,1,4.6,0,2.0,280.4,\bftab{4.6},280.4,4.6,\bftab{2.0}
grlex-13,92,741,0.3740,0.8000,1,1,18,1788,8.5,14037.2,8,10,1,1,1,4.1,0,2.3,14037.2,\bftab{4.1},14037.2,4.1,\bftab{2.3}
\end{filecontents*}
\csvreader[head to column names, late after line = \\]{summary-splitnbound-grlex.csv}%
			{instance=\name, n=\size, m=\edges }{%
				\name & \size & \edges & \opt &\lb & \ub & \sizeI & \solvedinroot & \bbnodesTotal & \timePreelim & \timeSPLIThighlighted & \firstguess & \dinkelbachsteps &
				\dinkelbachBbnodestotal & \timeDBhighlighted & \gurobiSolTime
			} \bottomrule
		\end{tabular}
		\caption{Comparison of Algorithm~\ref{alg:split-and-bound}
			(split \& bound) and Algorithm~\ref{alg:dinkelbach} (Dinkelbach) for grlex instances.}
		\label{tab:result-grlex}

		\bigskip 
		
		\begin{tabular}{l|r|r|r||r|r|r|r|r|r|r||r|r|r|r||r}
			\multicolumn{4}{@{}c@{}}{}
			& \multicolumn{7}{@{}c@{}}{Algorithm~\ref{alg:split-and-bound} (split \& bound)}
			& \multicolumn{4}{@{}c@{}}{Algorithm~\ref{alg:dinkelbach} (Dinkelbach)}
			& \multicolumn{1}{@{}c@{}}{Gurobi}
			\\\cmidrule(r){1-4}\cmidrule(lr){5-11}\cmidrule(lr){12-15}\cmidrule(l){16-16}%
			\head{Instance} & \multicolumn{1}{@{}c@{}|}{$n$} &\multicolumn{1}{@{}c@{}|}{$m$} & \multicolumn{1}{@{}c@{}||}{\,$h(G)$\,} & \multicolumn{1}{@{}c@{}|}{$\min\, \ell_k$} & \multicolumn{1}{@{}c@{}}{$u^*$}& \multicolumn{1}{|@{}c@{}|}{$\lvert \mathcal I\rvert$} & \head{solved\\in root}& \head{B\&B\\nodes} &  \head{Alg.~\ref{alg:pre-elimination}\\time (s)} & \head{total\\time (s)} & \multicolumn{1}{@{}c@{}|}{\,\head{first\\guess}\,} & \head{\# of\\steps} & \head{B\&B\\nodes} & \head{total\\time (s)} & \head{total\\time (s)}\\ \midrule
			\begin{filecontents*}{summary-splitnbound-grevlex.csv}
instance,n,m,lambdatwohalf,lb,ub,opt,sizeI,bbnodesTotal,timePreelim,timeTotal,morebbnodes,solvedinroot,firstguess,dinkelbachsteps,dinkelbachBbnodestotal,dinkelbachtimetotal,gurobiRelGap,gurobiSolTime,timeSPLIThighlighted,timeDBhighlighted,timeSPLIThighlighted2,timeDBhighlighted2,timeGUROBIhighlighted2
%grevlex-2,4,4,1.0000,1.0000,1.0000,1.0000,0,0,0.1,0.1,0,0,1.0000,1,1,0.0,-,-,0.1,\bftab{0.0},0.1,\bftab{0.0},-
%grevlex-3,7,11,1.0000,1.3333,1.3333,1.3333,0,0,0.1,0.1,0,0,1.3333,1,1,0.1,-,-,\bftab{0.1},\bftab{0.1},\bftab{0.1},\bftab{0.1},-
%grevlex-4,11,24,1.1217,1.6000,1.7500,1.7500,1,1,0.1,0.1,0,1,1.7500,1,1,0.7,-,-,\bftab{0.1},0.7,\bftab{0.1},0.7,-
%grevlex-5,16,45,1.2677,1.6250,1.8750,1.8750,2,2,0.2,0.3,0,2,1.8750,1,1,0.4,-,-,\bftab{0.3},0.4,\bftab{0.3},0.4,-
%grevlex-6,22,76,1.4280,1.8182,2.0000,2.0000,1,1,0.3,0.3,0,1,2.0000,1,1,0.3,-,-,\bftab{0.3},\bftab{0.3},\bftab{0.3},\bftab{0.3},-
grevlex-7,29,119,1.5969,2.1429,2.4615,2.4615,3,3,0.4,1.0,0,3,2.4615,1,41,33.3,0.000,1.1,\bftab{1.0},33.3,\bftab{1.0},33.3,1.1
grevlex-8,37,176,1.7715,2.3889,2.8333,2.8333,5,5,1.0,5.8,0,5,2.8333,1,105,188.6,0.000,3.7,\bftab{5.8},188.6,5.8,188.6,\bftab{3.7}
grevlex-9,46,249,1.9499,2.5652,2.9565,2.9565,5,5,1.5,20.7,0,5,2.9565,1,89,194.1,0.000,39.9,\bftab{20.7},194.1,\bftab{20.7},194.1,39.9
grevlex-10,56,340,2.1310,2.7857,3.2222,3.2222,6,6,2.9,33.8,0,6,3.2222,1,161,316.5,0.000,70.3,\bftab{33.8},316.5,\bftab{33.8},316.5,70.3
grevlex-11,67,451,2.3141,3.0909,3.6667,3.6667,8,20,3.5,193.9,1,7,3.7188,2,1478,10412.3,0.000,1460.7,\bftab{193.9},10412.3,\bftab{193.9},10412.3,1460.7
grevlex-12,79,584,2.4987,3.3333,3.9231,3.9231,9,241,6.9,1315.5,7,2,3.9231,1,1293,11861.7,0.000,8624.9,\bftab{1315.5},11861.7,\bftab{1315.5},11861.7,8624.9
grevlex-13,92,741,2.6845,3.5435,4.0000,4.0000,7,475,9.4,2246.3,6,1,4.0000,1,1,29.4,0.540,-,2246.3,\bftab{29.4},2246.3,\bftab{29.4},-
\end{filecontents*}
\csvreader[head to column names, late after line = \\]{summary-splitnbound-grevlex.csv}%
			{instance=\name, n=\size, m=\edges }{%
				\name & \size & \edges & \opt &\lb & \ub & \sizeI & \solvedinroot & \bbnodesTotal & \timePreelim & \timeSPLIThighlighted & \firstguess & \dinkelbachsteps &
				\dinkelbachBbnodestotal & \timeDBhighlighted & \gurobiSolTime
			} \bottomrule
		\end{tabular}
		\caption{Comparison of Algorithm~\ref{alg:split-and-bound}
			(split \& bound),
			Algorithm~\ref{alg:dinkelbach} (Dinkelbach)
			and Gurobi for grevlex instances.}
		\label{tab:result-grevlex}

	\end{sidewaystable}

	\begin{sidewaystable}
		\centering
		\tiny
		\begin{tabular}{l|r|r|r||r|r|r|r|r|r|r||r|r|r|r||r}
			\multicolumn{4}{@{}c@{}}{}
			& \multicolumn{7}{@{}c@{}}{Algorithm~\ref{alg:split-and-bound} (split \& bound)}
			& \multicolumn{4}{@{}c@{}}{Algorithm~\ref{alg:dinkelbach} (Dinkelbach)}
			& \multicolumn{1}{@{}c@{}}{Gurobi}
			\\\cmidrule(r){1-4}\cmidrule(lr){5-11}\cmidrule(lr){12-15}\cmidrule(l){16-16}%
			\head{Instance} & \multicolumn{1}{@{}c@{}|}{$n$} &\multicolumn{1}{@{}c@{}|}{$m$} & \multicolumn{1}{@{}c@{}||}{\,$h(G)$\,} & \multicolumn{1}{@{}c@{}|}{$\min\, \ell_k$} & \multicolumn{1}{@{}c@{}}{$u^*$}& \multicolumn{1}{|@{}c@{}|}{$\lvert \mathcal I\rvert$} & \head{solved\\in root}& \head{B\&B\\nodes} &  \head{Alg.~\ref{alg:pre-elimination}\\time (s)} & \head{total\\time (s)} & \multicolumn{1}{@{}c@{}|}{\,\head{first\\guess}\,} & \head{\# of\\steps} & \head{B\&B\\nodes} & \head{total\\time (s)} & \head{total\\time (s)}\\ \midrule
			\begin{filecontents*}{summary-splitnbound-dimacs.csv}
instance,n,m,lambdatwohalf,lb,ub,opt,sizeI,bbnodesTotal,timePreelim,timeTotal,morebbnodes,solvedinroot,firstguess,dinkelbachsteps,dinkelbachBbnodestotal,dinkelbachtimetotal,gurobiRelGap,gurobiSolTime,timeSPLIThighlighted,timeDBhighlighted,timeSPLIThighlighted2,timeDBhighlighted2,timeGUROBIhighlighted2
karate,34,78,0.2343,0.5000,0.5882,0.5882,4,4,0.7,2.3,0,4,0.5882,1,1,1.0,0,0.2,2.3,\bftab{1.0},2.3,1,\bftab{0.2}
chesapeake,39,170,0.8184,2.0000,2.1667,2.1667,8,8,1.0,2.0,0,8,2.1667,1,1,2.4,0,0.7,\bftab{2.0},2.4,2,2.4,\bftab{0.7}
dolphins,62,159,0.0865,0.2000,0.2857,0.2857,16,16,4.0,13.2,0,16,0.2857,1,1,2.3,0,0.7,13.2,\bftab{2.3},13.2,2.3,\bftab{0.7}
lesmis,77,254,0.1025,0.2500,0.3000,0.3000,2,2,4.7,14.7,0,2,0.3000,1,1,8.0,0,0.7,14.7,\bftab{8.0},14.7,8,\bftab{0.7}
polbooks,105,441,0.1618,0.3269,0.3654,0.3654,37,37,18.0,540.0,0,37,0.3654,1,11,128.7,0,3.3,540,\bftab{128.7},540,128.7,\bftab{3.3}
adjnoun,112,425,0.3475,1.0000,1.0000,1.0000,0,0,16.9,16.9,0,0,1.0000,1,1,8.6,0,4.2,16.9,\bftab{8.6},16.9,8.6,\bftab{4.2}
football,115,613,0.7295,0.9825,1.0702,1.0702,5,55,15.2,399.9,1,4,1.0702,1,1,25.8,0,31.2,399.9,\bftab{25.8},399.9,\bftab{25.8},31.2
jazz,198,2742,0.2860,1.0000,1.0000,1.0000,0,0,118.4,118.4,0,0,1.0000,1,1,56.9,0,12.9,118.4,\bftab{56.9},118.4,56.9,\bftab{12.9}
celegansneural,297,2148,0.4243,1.0000,1.0000,1.0000,0,0,389.3,389.3,0,0,1.0000,1,1,80.6,0,22.0,389.3,\bftab{80.6},389.3,80.6,\bftab{22.0}
celegans\_metabolic,453,2025,0.1290,0.3333,0.5000,0.4000,20,24,1475.6,2383.3,1,19,0.5000,3,3,828.2,0,5.1,2383.3,\bftab{828.2},2383.3,828.2,\bftab{5.1}
\end{filecontents*}
\csvreader[head to column names, late after line = \\]{summary-splitnbound-dimacs.csv}%
			{instance=\name, n=\size, m=\edges }{%
				\name & \size & \edges & \opt &\lb & \ub & \sizeI & \solvedinroot & \bbnodesTotal & \timePreelim & \timeSPLIThighlighted & \firstguess & \dinkelbachsteps &
				\dinkelbachBbnodestotal & \timeDBhighlighted & \gurobiSolTime
			} \bottomrule
		\end{tabular}
		\caption{Comparison of
		Algorithm~\ref{alg:split-and-bound} (split \& bound),
			Algorithm~\ref{alg:dinkelbach} (Dinkelbach)
		and Gurobi for DIMACS instances.}
		\label{tab:result-dimacs}
		
		\bigskip
		
		\begin{tabular}{l|r|r|r||r|r|r|r|r|r|r||r|r|r|r||r}
			\multicolumn{4}{@{}c@{}}{}
			& \multicolumn{7}{@{}c@{}}{Algorithm~\ref{alg:split-and-bound} (split \& bound)}
			& \multicolumn{4}{@{}c@{}}{Algorithm~\ref{alg:dinkelbach} (Dinkelbach)}
			& \multicolumn{1}{@{}c@{}}{Gurobi}
			\\\cmidrule(r){1-4}\cmidrule(lr){5-11}\cmidrule(lr){12-15}\cmidrule(l){16-16}%
			\head{Instance} & \multicolumn{1}{@{}c@{}|}{$n$} &\multicolumn{1}{@{}c@{}|}{$m$} & \multicolumn{1}{@{}c@{}||}{\,$h(G)$\,} & \multicolumn{1}{@{}c@{}|}{$\min\, \ell_k$} & \multicolumn{1}{@{}c@{}}{$u^*$}& \multicolumn{1}{|@{}c@{}|}{$\lvert \mathcal I\rvert$} & \head{solved\\in root}& \head{B\&B\\nodes} &  \head{Alg.~\ref{alg:pre-elimination}\\time (s)} & \head{total\\time (s)} & \multicolumn{1}{@{}c@{}|}{\,\head{first\\guess}\,} & \head{\# of\\steps} & \head{B\&B\\nodes} & \head{total\\time (s)} & \head{total\\time (s)}\\ \midrule
			\begin{filecontents*}{summary-splitnbound-network.csv}
instance,n,m,lambdatwohalf,lb,ub,opt,sizeI,bbnodesTotal,timePreelim,timeTotal,morebbnodes,solvedinroot,firstguess,dinkelbachsteps,dinkelbachBbnodestotal,dinkelbachtimetotal,gurobiRelGap,gurobiSolTime,timeSPLIThighlighted,timeDBhighlighted,timeSPLIThighlighted2,timeDBhighlighted2,timeGUROBIhighlighted2
moviegalaxies-567,52,146,0.1663,0.3636,0.3810,0.3810,3,3,2.3,3.5,0,3,0.3810,1,1,2.7,0,0.4,3.5,\bftab{2.7},3.5,2.7,\bftab{0.4}
moviegalaxies-52,59,119,0.1771,0.4000,0.5385,0.5385,27,27,3.9,16.3,0,27,0.5385,1,1,9.1,0,0.6,16.3,\bftab{9.1},16.3,9.1,\bftab{0.6}
terrorists-911,62,152,0.0919,0.2000,0.2174,0.2174,6,6,3.2,10.7,0,6,0.2174,1,1,7.5,0,0.5,10.7,\bftab{7.5},10.7,7.5,\bftab{0.5}
train\_terrorists,64,243,0.1652,0.4000,0.6000,0.6000,20,20,5.2,44.9,0,20,0.6000,1,1,2.5,0,1.1,44.9,\bftab{2.5},44.9,2.5,\bftab{1.1}
highschool,70,274,0.3975,0.7059,0.9143,0.9143,26,26,5.5,131.2,0,26,0.9143,1,9,92.1,0,1.7,131.2,\bftab{92.1},131.2,92.1,\bftab{1.7}
blumenau\_drug,75,181,0.1753,0.5000,0.5000,0.5000,0,0,5.1,5.1,0,0,0.5000,1,1,4.6,0,1.7,5.1,\bftab{4.6},5.1,4.6,\bftab{1.7}
sp\_office,92,755,1.6520,3.1739,3.3696,3.3696,5,5,9.9,19.3,0,5,3.3696,1,77,858.7,0,522.2,\bftab{19.3},858.7,\bftab{19.3},858.7,522.2
swingers,96,232,0.1325,0.3333,0.3333,0.3333,0,0,10.2,10.2,0,0,0.5000,3,3,26.2,0,1.4,\bftab{10.2},26.2,10.2,26.2,\bftab{1.4}
game\_thrones,107,352,0.1197,0.2857,0.4211,0.4000,22,22,13.0,290.6,0,22,0.4375,2,2,28.6,0,2.7,290.6,\bftab{28.6},290.6,28.6,\bftab{2.7}
revolution,141,160,0.0382,0.0770,0.0962,0.0962,33,111,39.4,1595.6,5,28,0.1000,2,98,639.0,0,1.3,1595.6,\bftab{639.0},1595.6,639,\bftab{1.3}
foodweb\_little\_rock,183,2434,0.4898,1.0000,1.0000,1.0000,0,0,99.2,99.2,0,0,1.0000,1,1,21.4,0,8.9,99.2,\bftab{21.4},99.2,21.4,\bftab{8.9}
cintestinalis,205,2575,0.4297,1.0000,1.0000,1.0000,0,0,117.9,117.9,0,0,1.0000,1,1,25.2,0,20.1,117.9,\bftab{25.2},117.9,25.2,\bftab{20.1}
malaria\_genes\_HVR1,307,2812,0.1032,0.2105,0.2377,0.2377,120,1890,503.1,62943.4,29,91,0.2377,1,5,425.8,0,7.8,62943.4,\bftab{425.8},62943.4,425.8,\bftab{7.8}
\end{filecontents*}
\csvreader[head to column names, late after line = \\]{summary-splitnbound-network.csv}%
			{instance=\name, n=\size, m=\edges }{%
				\name & \size & \edges & \opt &\lb & \ub & \sizeI & \solvedinroot & \bbnodesTotal & \timePreelim & \timeSPLIThighlighted & \firstguess & \dinkelbachsteps &
				\dinkelbachBbnodestotal & \timeDBhighlighted &  \gurobiSolTime
			} \bottomrule
		\end{tabular}
		\caption{Comparison of Algorithm~\ref{alg:split-and-bound}
			(split \& bound),
			Algorithm~\ref{alg:dinkelbach} (Dinkelbach)
			and Gurobi for network instances.}
		\label{tab:result-network}
	\end{sidewaystable}

\begin{figure}[hbt]
	\centering
%		\begin{subfigure}[b]{0.45\textwidth}
		%\resizebox{170pt}{150pt}{
		\begin{tikzpicture}[trim axis left]
		\footnotesize
		\begin{axis}[legend pos = north west, %legend style={at={(-0.3,0.6)},anchor=north east},
		xtick={0,100,200,300,400,500},
		ytick={0,0.2,0.4,0.6,0.8,1},
		enlarge x limits=false,
		enlarge y limits={upper},
		legend cell align={left},
		xlabel= time (sec),
		ylabel= \% of instances solved,
		ymin=0,
		ymax=1]
		\addplot+[red, mark=10-pointed star] table [x=time, y=gurobi]{data-performance-plot/performance-smalltimes.dat};
		\addlegendentry{\footnotesize Gurobi}
		\addplot+[green, mark options={fill=green}] table [x=time, y=dinkelbach]{data-performance-plot/performance-smalltimes.dat};
		\addlegendentry{\footnotesize Dinkelbach}
		\addplot+[blue, mark options={fill=blue}] table [x=time, y=splitandbound]{data-performance-plot/performance-smalltimes.dat};
		\addlegendentry{\footnotesize Split \& Bound}
		\end{axis}
		\end{tikzpicture}%
		\hfill
		%}
		%\caption{Time to solve the instances}
%	\end{subfigure}%
%\hfill
%	\begin{subfigure}[b]{0.45\textwidth}
		\centering
		%\resizebox{170pt}{150pt}{
		\begin{tikzpicture}[trim axis left]
		\footnotesize
		\begin{axis}[legend pos = north west, %legend style={at={(-0.3,0.6)},anchor=north east},
		xtick={2000,4000,6000,8000,10000},
		ytick={0,0.2,0.4,0.6,0.8,1},
		yticklabel pos=right,
		enlarge x limits=false,
		enlarge y limits={upper},
		legend cell align={left},
		xlabel= time (sec),
		ymin=0,
		ymax=1]
		\addplot+[red, mark=10-pointed star] table [x=time, y=gurobi]{data-performance-plot/performance-largetimes.dat};
		%\addlegendentry{\footnotesize Gurobi}
		\addplot+[green, mark options={fill=green}] table [x=time, y=dinkelbach]{data-performance-plot/performance-largetimes.dat};
		%\addlegendentry{\footnotesize Alg.~\protect\ref{alg:dinkelbach}}
		\addplot+[blue, mark options={fill=blue}] table [x=time, y=splitandbound]{data-performance-plot/performance-largetimes.dat};
		%\addlegendentry{\footnotesize Alg.~\protect\ref{alg:split-and-bound}}
		\end{axis}
		\end{tikzpicture}%}
		%\caption{Time to solve the instances}
%	\end{subfigure}%
	\caption{Performance comparison of the exact algorithms. Note
	the different scale on the $x$-axis: the plot on the left
	displays the time range from~0 to 500~seconds, the plot on the
	right from 500~seconds to 3~hours.} 
	\label{fig:performance}
\end{figure}
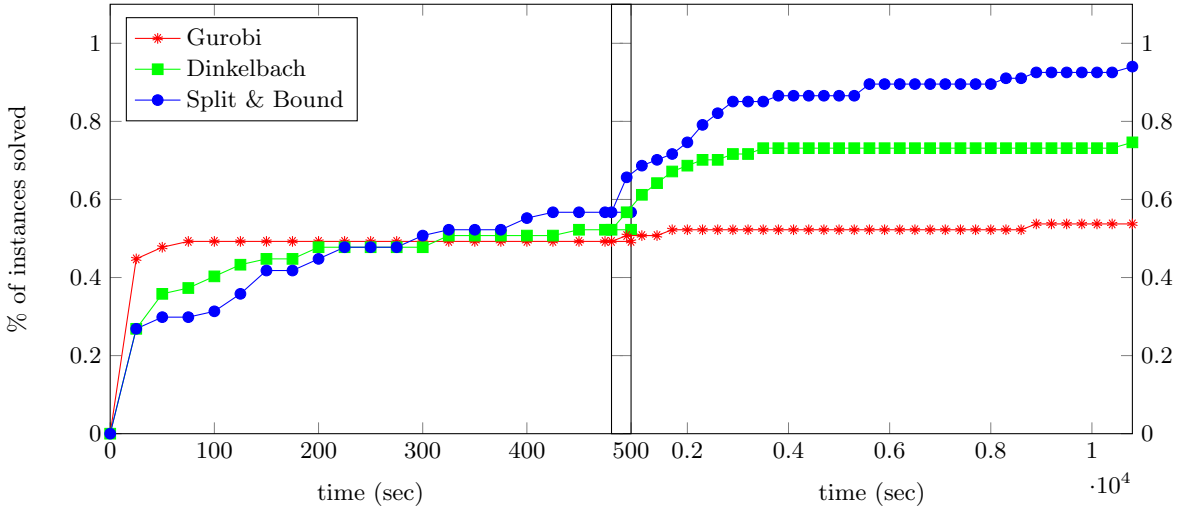

	\section{Summary and future research}
	\label{sec:summary}
	We developed a split \& bound algorithm as well as an algorithm
	applying Dinkelbach's idea for fractional programming to compute the edge
	expansion of a graph. The splitting refers to the fact, that we
	consider the different values of $k$ ($k$ being the size of the smaller
	partition) separately. We used semidefinite programming in both phases
	of our algorithm: on the one hand, SDP-based bounds are used to
	eliminate several values for~$k$ and we use an SDP-based 
	max-cut solver to 
	solve the problem for~$k$ fixed.
	Also, the algorithm following the Dinkelbach framework uses
	semidefinite programming in order to solve the underlying parametrized
	problems. 
	Through numerical results on various graph classes, we demonstrate
	that our split-and-bound algorithm is a robust method for
	computing the edge expansion while Dinkelbach's algorithm and
	Gurobi are very sensitive with respect to the edge expansion
	of the graph.
	
	In some applications, one is not interested in the exact value
	of the edge expansion but wants to check whether a certain
	value is a lower bound on the edge expansion, e.g., to check
	the Mihail-Vazirani conjecture on 0/1-polytopes. We
	implemented an option in our algorithm that enables this
	feature of verifying a given lower bound.

        As a heuristic, we use a simulated annealing approach that works very
	well for the problem sizes we are interested in. However, if one wants
	to obtain high-quality solutions for larger instances, a more
	sophisticated heuristic will be needed. Tabu-search, genetic
	algorithms, or a heuristic in the spirit of the Goemans-Williamson
	rounding could be potential candidates. In future research, we will also investigate convexification techniques by using recent results on fractional programming \cite{he2023convexification} and on exploiting submodularity of the cut function as has been done for mixed-binary conic optimization \cite{atamturk2020submodularity}.
	
	\bibliographystyle{spmpsci}
	\bibliography{expansion}
	
	\appendix
        
        \section{Computing the maximum cut of a graph}\label{sec:biqbin}
	Some of our algorithms for computing $\h$ rely on finding the maximum
	cut in a graph. For computing the value of the max-cut, we will use
	the SDP-based solver BiqBin~\cite{biqbin}. Note that the software
	BiqBin can not only compute the max-cut in a graph and solve
	quadratic unconstrained binary problems (QUBOs) but it is also
	applicable to linearly constrained binary problems with a quadratic
	objective function.
	However, we only need the max-cut solver in this work, and briefly
	describe the main ingredients in this section.
	
	BiqBin is a branch-and-bound algorithm that uses a tight SDP relaxation
	as upper bound and the celebrated Goemans-Williamson rounding
	procedure to generate a high-quality lower bound on the value of the
	maximum cut in a graph. To be more precise, the SDP   
	\begin{equation}\label{eq:maxcutsdp}
	\max_{\mat X}\, \Big\{ \frac 1 4 \inprod{\mat L}{\mat X} : \diag(\mat X)
	= \vec e,\, \mathcal{A}(\mat X) = \vec b,\,  \mat X \succeq 0 \Big\}
	\end{equation}
	where $\mathcal{A}(\mat X) \le \vec b$ models a set of triangle-,
	pentagonal- and heptagonal-inequalities is approximately solved using
	a bundle method. To do so, only the inequality constraints are dualized
	yielding the nonsmooth convex  partial dual function 
	\begin{align*}
	f(\vec \gamma) &= \max_{\mat X} \Big\{ \frac 1 4 \langle  \mat L, \mat X\rangle
	- \vec \gamma^\top( \mathcal{A}(\mat X) - \vec b) : \diag(\mat X) \
	= \ \vec e,\, \mat X \succeq 0\Big\}\\
	&= \vec b^\top \vec \gamma + \max_{\mat X} \Big\{\Big\langle \frac 1 4 \mat L - \mathcal{A}^\top
	(\vec \gamma), \mat X \Big\rangle : \diag(\mat X) 
	= \vec e,\, \mat X \succeq 0\Big\}
	\end{align*}
	where $\vec \gamma$ are the nonnegative dual variables associated with
	the constraints $\mathcal{A}(\mat X)\le \vec b$.
	Evaluating the dual function $f(\vec \gamma)$ and computing the subgradient amounts to solving an SDP 
	that can be efficiently computed using an interior-point method tailored for this problem.
	It provides us with the matching pair $(\mat
	X_{\gamma}, \vec \gamma)$
	such that $f(\vec \gamma) = \vec b^\top \vec \gamma + \langle \nicefrac 1 4 \mat L - \mathcal{A}^\top
	(\vec \gamma), \mat X_\gamma \rangle$.
	Moreover, the subgradient of $f$ at $\vec \gamma$ is given by $\partial
	f(\vec \gamma) = \vec b - \mathcal{A}(\mat X_{\gamma}).$
	For obtaining an approximate minimizer of problem
	\begin{equation*}
	\min_{\vec \gamma} \{ f(\vec \gamma) :  \vec \gamma \ge 0\},
	\end{equation*}
	the bundle method is used. We refer to~\cite{ReRiWi:10} for more details.

	Note that interior-point 
	methods are far from computing a solution of \cref{eq:maxcutsdp}
	already for small graphs due to the number of constraints being too
	large and therefore already forming the system matrix is an expensive
	task or even impossible due to memory requirements. 
	
	BiqBin dominates all max-cut solvers based on linear programming and
	is comparable to the SDP-based solver
	BiqCrunch~\cite{krislock2017biqcrunch}.
	Moreover, BiqBin is available as a parallelized version.

\end{document}